\documentclass[a4paper, reqno, 11pt]{amsart}
\usepackage{amssymb,amsmath,amsthm,amsfonts,color} 

\usepackage[left=3.0cm,right=3.0cm,top=2.5cm,bottom=2.5cm]{geometry}

\usepackage{graphicx}
\usepackage{setspace}
\usepackage{comment}

\usepackage[colorlinks=true, pdfstartview=FitV, linkcolor=blue, 
            citecolor=blue, urlcolor=blue]{hyperref}

\numberwithin{equation}{section}

\newcommand{\hamiltonian}{{\bf h}}
\newcommand{\potential}{{\bf v}}
\newcommand{\dispersion}{\mathfrak{e}}

\newcommand{\low}{\mathfrak{l}}
\newcommand{\up}{\mathfrak{u}}

\newcommand{\singpart}{\mathfrak{j}}

\newcommand{\p}{\partial}
\renewcommand{\epsilon}{\varepsilon}
\renewcommand{\S}{\mathbb{S}}
\renewcommand{\hat}{\widehat}

\newcommand{\C}{\mathbb{C}}

\newcommand{\R}{\mathbb{R}}
\newcommand{\T}{\mathbb{T}}
\newcommand{\cH}{\mathcal{H}}
\newcommand{\Z}{\mathbb{Z}}

\newcommand{\N}{\mathbb{N}}

\newcommand{\cF}{\mathcal{F}}

\newcommand{\disc}{{\ensuremath{\mathrm{disc}}}}
\newcommand{\ess}{{\ensuremath{\mathrm{ess}}}}

\newcommand{\no}{\nonumber}

\numberwithin{equation}{section}

\theoremstyle{plain}
\newtheorem{theorem}{Theorem}[section]
\newtheorem{proposition}[theorem]{Proposition}
\newtheorem{lemma}[theorem]{Lemma}

\theoremstyle{definition}

\newtheorem{remark}[theorem]{Remark}

\date{\today}

\begin{document}
\title[]{Expansion of eigenvalues of rank-one perturbations  of the discrete bilaplacian}

\author[A. Khalkhuzhaev] {Ahmad Khalkhuzhaev} 
\address[Ahmad Khalkhuzhaev]{Samarkand State University\\ 
University boulevard 3\\ 140104 Samarkand (Uzbekistan)}
\email[A. Khalkhuzhaev]{ahmad\_x@mail.ru}

\author[Sh. Kholmatov] {Shokhrukh Yu. Kholmatov} 
\address[Shokhrukh Yu. Kholmatov]{University of Vienna\\ Oskar-Morgenstern Platz 1\\1090 Vienna  (Austria)}
\email[Sh. Kholmatov]{shokhrukh.kholmatov@univie.ac.at}

\author[M. Pardabaev] {Mardon  Pardabaev} 
\address[Mardon Pardabaev]{Samarkand State University\\ 
University boulevard 3\\ 140104 Samarkand (Uzbekistan)}
\email[M. Pardabaev]{p\_mardon75@mail.ru}

\begin{abstract}
We consider the family 
$$
\hat \hamiltonian_\mu:=\hat\varDelta\hat \varDelta - \mu \hat \potential,\qquad\mu\in\R,
$$ 
of discrete Schr\"odinger-type operators in $d$-dimensional lattice $\Z^d$, 
where $\hat \varDelta$ is the discrete Laplacian and $\hat\potential$ 
is of rank-one.  We prove that there exist coupling constant thresholds $\mu_o,\mu^o\ge0$ such that for any $\mu\in[-\mu^o,\mu_o]$ the discrete spectrum of $\hat \hamiltonian_\mu$ is empty and for any $\mu\in \R\setminus[-\mu^o,\mu_o]$ the discrete spectrum of $\hat\hamiltonian_\mu$ is a singleton $\{e(\mu)\},$ and  $e(\mu)<0$ for $\mu>\mu_o$ and  $e(\mu)>4d^2$ for $\mu<-\mu^o.$ Moreover, we study the asymptotics of $e(\mu)$ as $\mu\to\mu_o$ and $\mu\to -\mu^o$ as well as $\mu\to\pm\infty.$ The asymptotics highly depend  on $d$ and $\hat\potential.$ 
\end{abstract}



\maketitle

\begin{comment}
Possible journals:

1. Integral equations and operator theory (Ch. Tretter)\\
2. Archiv der Mathematik (for short version)\\
3. Journal of Physics A: Mathematical and Theoretical\\

\end{comment}

\section{Introduction}
Fourth order elliptic operators in $\R^d$ in particular, the biharmonic operator, play a central role in a wide class of physical models such as linear elasticity theory,  rigidity problems (for instance, construction of suspension bridges) and in streamfunction formulation of Stokes flows  (see e.g. \cite{DKV:2018,MZ:2016,MW:1987} and references therein). Moreover, recent investigations have shown that the Laplace and biharmonic operators have high potential in image compression with the optimized and sufficiently sparse stored data \cite{HPW:2015}.  
The need for corresponding numerical simulations has led to a vast literature devoted to a variety of discrete approximations to the solutions of fourth order equations \cite{BK:2019,GHKW:2018,Tee:1963}. The question of stability of such models is basically related to their spectral properties and therefore, numerous studies have been dedicated to the numerical evaluation of the eigenvalues \cite{AP:1986,Bou:2003,RB:2013}. 

In this paper we investigate the spectral properties of the discrete biharmonic operator $\hat \varDelta\hat\varDelta$ perturbed by a rank-one potential $\hat\potential,$ i.e. 
$$
\hat \hamiltonian_\mu:=\hat\varDelta\hat \varDelta - \mu \hat \potential 
$$
in the $d$-dimensional cubical lattice $\Z^d,$ where $\mu\in\R$ is a coupling constant.   
This model includes also a discrete Schr\"odinger operator on $\Z^d,$ associated to a system of one particle whose dispersion relation has a degenerate bottom. 
Moreover, in view of the momentum representation  \eqref{hamilt_momentum} below, $\hat\hamiltonian_\mu$ can also be considered as a Friedrichs model in $L^2(\T^d)$ with degenerate bottom, where $\T^d$ is the $d$-dimensional torus.
We recall that spectral theory of discrete Schr\"odinger operators and Friedrichs models with non-degenerate bottom in particular, with discrete Laplacian, have been extensively studied in recent years (see e.g. \cite{ALM:04:Puan,ALMM:2006,GSch:97,LKL:2012,LH:2011,Mat,Mog} and references therein) because of their applications in the theory of ultracold atoms in optical lattices \cite{JBC:1998,LSA:2012,Wall:2015,Wink:2006}. In these models the appearance of weakly coupled bound states has been sufficiently well-explored and necessary and sufficient conditions for the existence of discrete spectrum in terms of the coupling constant have been established. This conditions naturally leads to the coupling constant threshold phenomenon \cite{KS:1980}: consider $-\Delta +\lambda V$ with $V$ short range  at a value $\lambda_0,$ where some eigenvalue $e(\lambda)$ is absorbed into the continuous spectrum as $\lambda\searrow \lambda_0,$ and conversely, for any $\epsilon>0,$ as $\lambda\nearrow \lambda_0+\epsilon$ the continous spectrum gives birth to a new eigenvalue. This phenomenon and the absorbtion  rate of eigenvalues as $\lambda\to\lambda_0$ have been already established for discrete Schr\"odinger operators with non-degenerate bottom in \cite{LKL:2012,LH:2011,LH:2012}. 

In the case of Schr\"odinger operators with degenerate bottom some sufficient conditions for existence of bound states have been recently observed in \cite{HHV:2018}. However, to the best of our knowledge, except the case $d=1$ in \cite{HP:2019},  the results related to the asymptotics of eigenvalues as coupling constant approaches to the threshold have not been published yet.

The threshold phenomenon is also closely related to the low-energy and high-energy resonances, and hence, to the Efimov effect \cite{ALM:04:Puan,BT:2017,Gr:2014,Lak:1993,NE:2017,Sob:1993,Tam:1991,Yaf:1974}: {\it if any two-body subsystem in a three-body system has no bound state below its essential spectrum and at least two two-body subsystem has a zero-energy resonance, then the corresponding  three-body system has infinitely many bound states whose energies accumulate at the lower edge of the three-body essential spectrum}. 
In fact, it is well-known that the two-body resonances can be observed exactly at the coupling constant thresholds (see e.g. \cite{ALM:04:Puan}). In the continuous case, the Efimov effect may appear only for certain attractive systems of particles \cite{NE:2017}, and in the discrete case an analogous result has been established for attractive three-particle system in three dimensional lattice.  Recent experimental results in the theory of ultracold atoms in an optical lattice have shown that two-particle systems can have repulsive  bound states and resonances (see e.g. \cite{Wink:2006}), thus, one expects the Efimov effect to hold also for some repulsive three-particle systems.

In the current paper we study the spectrum of $\hat \hamiltonian_\mu$ and its dependence on the coupling constant $\mu$ and on the perturbation $\hat \potential.$ For simplicity we assume the generator $\hat v$ of $\hat\potential$ to decay exponentially at infinity, however, we urge that our methods can also be adjusted to less regular cases (see Remark \ref{rem:comments}). 
In view of the momentum representation \eqref{hamilt_momentum}, the discrete bilaplacian is unitarily equivalent to the multiplication operator in $L^2(\T^d)$ by the analytic function 
$
\dispersion(p):=(\sum\limits_{i=1}^d (1-\cos p_i))^2 
$ 
so that the discrete bilaplacian has only the essential spectrum filling in the segment $[0,4d^2],$ and hence, by the compactness of  $\hat \potential$ and Weyl's Theorem, $\sigma_\disc(\hat \hamiltonian_\mu)=[0,4d^2]$ independently of $\mu.$ Moreover, the rank-one property of $\hat \potential$ implies that  for any $\mu\in\R$ there is at most one eigenvalue $e(\mu)$ of $\hat \hamiltonian_\mu $ outside its essential spectrum. Depending on the sign of $\mu,$ this eigenvalue can appear either below or above the essential spectrum, hence, we have two related threshold phenomena (see Theorem \ref{teo:existence_disc_spec}): the essential spectrum does not give birth to a new eigenvalue while $\mu$ runs in the interval $[-\mu^o,\mu_o],$ where  
\begin{equation}\label{muo_va_muo}
\mu_o:= \Big(\int_{\T^d} \frac{|v(q)|^2dq}{\dispersion(q)}\Big)^{-1}\ge0,\qquad \mu^o:=   \Big(\int_{\T^d} \frac{|v(q)|^2dq}{4d^2 - \dispersion(q)}\Big)^{-1}\ge0, 
\end{equation}
and $v(q)$ is the Fourier image of the the generator $\hat v$ of $\hat\potential,$   and as soon as $\mu$ leaves this interval through $\mu_o$ resp. through $-\mu^o,$ a negative resp. positive eigenvalue $e(\mu)$ comes out of the essential spectrum.  Conversely, if $\mu$ decreases to $\mu_o$ resp. increases to $-\mu^o,$ this unique eigenvalue $e(\mu)$ is absorbed by the essential spectrum, therefore, $-\mu^o$ and $\mu_o$ are the coupling constant thresholds. 

Now we are interested in the absorption rate of $e(\mu)$ as $\mu\to\mu_o$ and $\mu\to-\mu^o.$ In view of \eqref{muo_va_muo}, $\mu_o$ and $\mu^o$ depend not only on the dimension $d$ of the lattice, but also values $v(0)$ and $v(\vec\pi)$ of $v$ at the minimum and maximum points $0\in\T^d$ and $\vec\pi\in\T^d$ of $\dispersion:$  because of the singularity relaxation of $\frac{|v(q)|^2}{\dispersion(q)}$ resp. $\frac{|v(q)|^2}{4d^2 - \dispersion(q)},$ given $d\ge1$ the asymptotics of $e(\mu)$ with $v(0)=0$ resp. $v(\vec\pi)=0$ may completely differ from the case $v(0)\ne0$ resp. $v(\vec\pi)\ne 0.$ 
Furthermore, observing that the top $\dispersion(\vec\pi)=4d^2$ of the essential spectrum is non-degenerate, one expects the asymptotics of $e(\mu)$ as $\mu\to-\mu^o$ to be similar as in the discrete Laplacian case \cite{LKL:2012,LH:2011}; more precisely, depending on $d$ and the multiplicity $2n^o$ of $\vec\pi\in\{v=0\}$ (if $v(\vec\pi)=0$), 
$e(\mu)$ has a convergent expansion 
\begin{itemize}
\item[--]  in $\mu+\mu^o$ for $2n^o+d=1,3;$ 

\item[--] in $(\mu+\mu^o)^{1/2}$ for $2n^o+d\ge5$ with $d$ odd;

\item[--] in $\mu+\mu^o$ and $e^{-1/(\mu+\mu^o)}$ for $2n_o+d=2;$

\item[--] in  $\mu+\mu^o,$ $-\frac{1}{\ln(\mu+\mu^o)}$ and $-\frac{\ln\ln(\mu+\mu^o)^{-1}}{\ln(\mu+\mu^o)}$ for $2n_o+d=4;$

\item[--] in  $\mu+\mu^o$ and $-(\mu+\mu^o)\ln(\mu+\mu^o)$ for $2n^o+d\ge6$ with $d$ even 
\end{itemize}
(see Theorem \ref{teo:exp_e_atmax}). Moreover, the virtual states of energy $4d^2$, i.e. non-zero solutions $f$ of $\hat\hamiltonian_{-\mu^o}f=4d^2f$ not belonging to $\ell^2(\Z^d)$ appear  if and only if $2n_o+d=3,4.$ 
 However, the bottom $\dispersion(0)=0$ of the essential spectrum is degenerate, therefore, the asymptotics of $e(\mu)$ changes substantially. More precisely,  depending on $d$ and the multiplicity $2n_o$ of $0\in\{v=0\}$ (if $v(0)=0$),  $e(\mu)$ has a convergent expansion 
\begin{itemize}
\item[--]  in $(\mu-\mu_o)^{1/3}$ for $2n_o+d=1,7;$ 

\item[--]  in $\mu-\mu_o$ for $2n_o+d=3,5;$ 

\item[--]  in $(\mu-\mu_o)^{1/4}$ for $2n_o+d\ge9$ with $d$ odd; 

\item[--] in $\mu- \mu_o$ and $-(\mu- \mu_o)\ln(\mu- \mu_o)$ for $2n_o+d=2,6;$

\item[--] in $\mu-\mu_o$ and $e^{-1/(\mu-\mu_o)}$ for $2n_o+d=4;$

\item[--] in  $(\mu-\mu_o)^{1/2},$ $-(\mu-\mu_o)\ln(\mu-\mu_o),$ $(-\frac{1}{\ln(\mu-\mu_o)})^{1/2}$ and $-\frac{\ln\ln(\mu-\mu_o)^{-1}}{\ln(\mu-\mu_o)}$ for $2n_o+d=8;$

\item[--] in  $(\mu-\mu_o)^{1/2}$ and $-(\mu-\mu_o)^{1/2}\ln(\mu-\mu_o)$ for $2n^o+d\ge10$ with $d$ even 
\end{itemize}
(see Theorem \ref{teo:exp_e_at0}). Moreover, virtual states of $0$-energy, i.e. non-zero solutions of $\hat \hamiltonian_{\mu_o}f=0$ not belonging to $\ell^2(\Z^d),$ appear if and only if $2n_o+d\in\{5,6,7,8\}.$ The emergence of $0$-energy resonances in more lattice dimensions could allow the Efimov effect to be observed in other dimensions than $d=3$ in systems of particles with degenrate dispersion relation.

The paper is organized as follows. In Section \ref{sec:premilinary} 
after introducing some preliminaries we state the main results of the paper.  
In Theorem \ref{teo:existence_disc_spec} we establish necessary and sufficient conditions for non-emptiness of the discrete spectrum of $\hat\hamiltonian_\mu,$ and in case of existence, we study the location and the uniqueness, analiticity, monotonicity and convexity properties of eigenvalues $e(\mu)$ as a function of $\mu.$ In particular, we study the asymptotics of $e(\mu)$   as $\mu\to\mu_o$ and $\mu\to-\mu^o$ as well as $\mu\to\pm\infty.$ As discussed above in Theorems \ref{teo:exp_e_at0} and \ref{teo:exp_e_atmax} we obtain expansions of $e(\mu)$ for small and positive $\mu-\mu_o$ and $\mu+\mu^o$. 
In Section \ref{sec:proof_mains} we prove the main results. The main idea of the proof is to obtain a nonlinear equation $\Delta(\mu;z)=0$ with respect to the eigenvalue $z=e(\mu)$ of $\hat \hamiltonian_\mu$ and then study properties of $\Delta(\mu;z).$ Finally, in appendix Section \ref{sec:asymp_integr} we obtain the asymptotics of certain integrals related to $\Delta(\mu;z)$ which will be used in the proofs of main results.

\section{Preliminary and statement of main results}\label{sec:premilinary}

Let $\Z^d$ be the $d$-dimensional lattice and $\ell^2(\Z^d)$ be the Hilbert space of square-summable functions on $\Z^d.$
Consider the family
$$
\hat \hamiltonian_\mu:=\hat\hamiltonian_0 - \mu\hat\potential,\qquad \mu\ge0,
$$
of self-adjoint bounded discrete Schr\"odinger operators in
$\ell^2(\Z^d).$ Here $\hat \hamiltonian_0$ is discrete bilaplacian, i.e. 
$$
\hat \hamiltonian_0:=\hat \varDelta\hat\varDelta,
$$
where
$$
\hat\varDelta f(x) = \frac12\,\sum\limits_{|s|=1} [f(x) - f(x+s)],\qquad f\in \ell^2(\Z^d),
$$
is the discrete Laplacian
and $\hat \potential$ is a rank-one operator 
$$
\hat \potential \hat f(x) = \hat v(x) \sum\limits_{y\in\Z^d}   {\hat v(y)} \hat f(y),
$$
where $\hat v\in \ell^2(\Z^d)\setminus\{0\}$  
is a real-valued function such that 
\begin{equation}\label{hat_v_ga_shart}
|\hat v(x)|\le Ce^{-a|x|},\qquad x\in\Z^d, 
\end{equation}
for some $C,a>0.$ Notice that under assumption \eqref{hat_v_ga_shart},  the function 
\begin{equation}\label{v_p}
v(p):=\frac{1}{2\pi}\sum\limits_{x\in\Z^d} \hat v(x)e^{ix\cdot p} 
\end{equation}
is analytic on $\T^d.$

Recalling  
$$
\sigma(\hat \varDelta) = \sigma_\ess(\hat \varDelta) =[0,2d] 
$$
from the spectral theory  it follows that 
$$
\sigma(\hat \hamiltonian_0) = \sigma_\ess(\hat \hamiltonian_0) =[0,4d^2],
$$
and hence, by the compactness of $\hat \potential $ and   Weyl's Theorem,
$$
\sigma_\ess(\hat \hamiltonian_\mu) = \sigma_\ess(\hat \hamiltonian_0) =[0,4d^2]
$$
for any $\mu\in\R.$ 

Let   
$$
\cF:\ell^2(\Z^d)\to L^2(\T^d),\qquad \cF\hat f(p) = \frac{1}{(2\pi)^{d/2}}\,\sum\limits_{x\in\Z^d} \hat f(x) e^{ixp} 
$$
be the standard Fourier transform with the inverse 
$$
\cF^{-1}:L^2(\T^d)\to \ell^2(\Z^d),\qquad \cF^{-1}f(x) = \frac{1}{(2\pi)^{d/2}}\,\int_{\T^d} f(p) e^{-ixp}dp. 
$$
The operator $\hamiltonian_\mu = \cF\hat \hamiltonian_\mu\cF^{-1}$ is called the {\it momentum representation} of $\hat\hamiltonian_\mu.$ Note that 
\begin{equation}\label{hamilt_momentum}
\hamiltonian_\mu =\hamiltonian_0 - \mu\potential, 
\end{equation}
where $\hamiltonian_0:= \cF\hat \hamiltonian_0\cF^{-1}$ is the multiplication operator in $L^2(\T^d)$ by 
$\dispersion(\cdot)$ defined in \eqref{dispersion_q},
and since $\hat v$ is even, $\potential: = \cF\hat \potential \cF^{-1}$  is the rank-one integral operator 
$$
\potential f(p) = v(p) \int_{\T^d} \overline{v(q)}f(q)dq,
$$
where $v(\cdot)$ is given by \eqref{v_p}. %Note that for $d=1$ and $\hat v$ is of zero-range, 

Now we state the main results of the paper.
First we concern with the existence of the discrete spectrum of $\hamiltonian_\mu.$ To this aim let 
$$
\mu_o:= \Big(\int_{\T^d} \frac{|v(q)|^2dq}{\dispersion(q)}\Big)^{-1},\qquad \mu^o:=   \Big(\int_{\T^d} \frac{|v(q)|^2dq}{4d^2 - \dispersion(q)}\Big)^{-1},
$$
where 
\begin{equation}\label{dispersion_q}
\dispersion(q): =\Big(\sum\limits_{i=1}^d ( 1 - \cos q_i)\Big)^2, 
\end{equation}

Note that $\mu_o,\mu^o\ge0.$ 

\begin{theorem} \label{teo:existence_disc_spec}
For any $\mu\in[-\mu^o,\mu_o]$ the discrete spectrum of $\hamiltonian_\mu$ is empty and
for any $\mu\in\R\setminus [-\mu^o,\mu_o]$ $\hamiltonian_\mu$ has a unique eigenvalue $e(\mu)$ outside the essential spectrum $[0,4d^2]$ with the associated eigenfunction 
\begin{equation}\label{eigenfunction}
f_\mu(p) = \frac{v(p)}{\dispersion(p) - e(\mu)}. 
\end{equation}
Moreover, 
if $\mu<-\mu^o$ resp. $\mu>\mu_o,$ then $e(\mu)>4d^2$ resp. $e(\mu)<0,$ the function
$\mu\in\R\setminus[-\mu^o,\mu_o]\mapsto  e(\mu)$ is real-analytic strictly decreasing, convex in $(-\infty,-\mu^o)$ and concave in $(\mu_o,+\infty)$ 
satisfying
\begin{equation}\label{asymptotics_e_chegara}
\lim\limits_{\mu\searrow \mu_o} e(\mu) =0\qquad \text{and}\qquad 
\lim\limits_{\mu\nearrow -\mu^o} e(\mu) =4d^2  
\end{equation}
and 
\begin{equation}\label{asymptotics_e_infty}
\lim\limits_{\mu\to\pm\infty} \frac{e(\mu)}{\mu} = \mp\int_{\T^d}|v(q)|^2dq.
\end{equation}
\end{theorem}

Now we study the rate of the convergences in  \eqref{asymptotics_e_chegara}.

\begin{theorem}[\textbf{Expansions of $e(\mu)$  at $\mu=\mu_o$}]\label{teo:exp_e_at0}
Let $n_o\ge0$ be such that 
\begin{equation*} 
|v(0)|^2=D^2|v(0)|^2 = \ldots=D^{2n_o-2}|v(0)|^2 = 0,\qquad D^{2n_o}|v(0)|^2\ne0,
\end{equation*}
where $D^jf(p)$ is the $j$-th order differential of $f$ at $p,$ i.e. the $j$-th order symmetric tensor 
$$
D^jf(p)[\underbrace{w,\ldots,w}_{j-times}] = \sum\limits_{i_1+\ldots +i_d =j,i_k\ge0} 
\frac{\p^jf(p)}{\p^{i_1} p_1\ldots \p^{i_d}p_d} \,w_1^{i_1}\ldots w_d^{i_d},\quad w=(w_1,\ldots,w_d)\in\R^d.
$$ 
Then 
$$
c_v:=\frac{2^{2n_o+d}}{(2n_o)!}\, \int_{\S^{d-1}} D^{2n_o}|v(0)|^2[w,\ldots, w]\,d\cH^{d-1}(w)>0
$$
and 
$$
\hat c_v:= \int_{\T^d} \frac{v(q)^2dq}{\dispersion(q)^2}\in(0,+\infty].
$$
Moreover, if $2n_o+d\ge5,$ then the equation $\hamiltonian_{\mu_o} f=0$ has a non-zero solution 
$$
f(p) = \frac{v(p)}{\dispersion(p)} 
$$
and $f\notin  L^2(\T^d)$ for $2n_o+d\in \{5,6,7,8\},$ and 
$f\in L^2(\T^d)$ for $2n_o+d\ge9.$

Finally, for $\mu>\mu_o$ let $e(\mu)<0$ be the eigenvalue of $\hamiltonian_\mu.$ 

\begin{itemize}
\item[(a)] Suppose that $d$ is odd:
\begin{itemize}
\item[(a1)] if $2n_o+d=1,3,$ then $\mu_o=0$ and for sufficiently small and positive $\mu,$
$$
(- e(\mu))^{1/4}= 
\begin{cases}
c_1\mu^{1/3} + \sum\limits_{n\ge1} c_{1,n} \mu^{\frac{n+1}{3}}, & 2n_o+d=1,\\
c_3\mu + \sum\limits_{n\ge1} c_{3,n} \mu^{n+1}, & 2n_o+d=3,
\end{cases}
$$
where $\{c_{1,n}\}$ and $\{c_{3,n}\}$ are some real coefficients and 
\begin{equation}\label{c1_c3}
c_1:=\Big(\frac{\pi c_v}{4}\Big)^{1/3}\qquad \text{and} \qquad c_3:=\frac{\pi c_v}{8};  
\end{equation}

\item[(a2)] if $2n_o+d=5,7,$ then $\mu_o>0$ and for sufficiently small and positive $\mu-\mu_o,$
$$
(- e(\mu))^{1/4}= 
\begin{cases}
c_5 (\mu-\mu_o) +\sum\limits_{n\ge1} c_{5,n} (\mu-\mu_o)^{n+1}, & 2n_o+d=5,\\
c_7 (\mu-\mu_o)^{1/3} + \sum\limits_{n\ge1} c_{7,n} (\mu-\mu_o)^{\frac{n+1}{3}}, & 2n_o+d=7, 
\end{cases}
$$
where $\{c_{5,n}\}$ and $\{c_{7,n}\}$ are some real coefficients and  
\begin{equation}\label{c5_c7}
c_5:=\frac{8}{\pi c_v\mu_o^2}  \qquad \text{and} \qquad c_7:=\Big(\frac{8}{\pi c_v\mu_o^2}\Big)^{1/3};  
\end{equation}

\item[(a3)] if $2n_o+d\ge9,$ then $\mu_o>0$ and for sufficiently small and positive $\mu-\mu_o,$
$$
(- e(\mu))^{1/4}= c_9(\mu-\mu_o)^{1/4} + \sum\limits_{n\ge1} c_{9,n} (\mu-\mu_o)^{n/4},  
$$
where $\{c_{9,n}\}$  are some real coefficients  and 
\begin{equation}\label{c9}
 c_9:=(\mu_o^2\hat c_v)^{-1/4}.
\end{equation}

\end{itemize}

\item[(b)] Suppose that $d$ is even:
\begin{itemize}
\item[(b1)] if $2n_o+d=2,4,$ then $\mu_o=0$ and for sufficiently small and positive $\mu,$
$$
(- e(\mu))^{1/2}= 
\begin{cases}
c_2\mu + \sum\limits_{n+m\ge1,n,m\ge0} c_{2,nm} \mu^{n+1}(-\mu\ln\mu)^{m}, & 2n_o+d=2,\\
ce^{-\frac{1}{c_4\mu}} +\sum\limits_{n+m\ge1,n,m\ge0} c_{4,nm} \mu^{n+1}\big(\frac1\mu\,e^{-\frac{1}{c_4\mu}}\big)^{m+1}, & 2n_o+d=4,
\end{cases}
$$
where $\{c_{2,nm}\}$ and $\{c_{4,nm}\}$ are some real coefficients,  $c>0,$ and
\begin{equation}\label{c2_c4}
c_2:= \frac{\pi c_v}{8}\qquad\text{and}\qquad c_4:= \frac{c_v}{8}; 
\end{equation}

\item[(b2)] if $2n_o+d=6,8,$ then $\mu_o>0$ and for sufficiently small and positive $\mu-\mu_o,$
$$
(- e(\mu))^{1/2}= 
\begin{cases}
c_6\tau^2 + \sum\limits_{n+m\ge1,n,m\ge0} c_{6,nm} \tau^{2n+2}\theta^m, & 2n_o+d=6,\\
c_8 \tau\sigma+ \sum\limits_{n+m+k\ge1,n,m,k\ge0} c_{8,nmk} \tau^{n+1} \sigma^{m+1}\eta^k, & 2n_o+d=8,
\end{cases}
$$
where $\{c_{4,nm}\}$ and $\{c_{8,nmk}\}$ are some real coefficients,
\begin{equation}\label{new_variables}
\begin{gathered}
\tau:=(\mu-\mu_o)^{1/2},\,\,\, \theta:=-\tau^2\ln\tau,\,\,\,
\sigma :=\Big(-\frac{1}{\ln\tau}\Big)^{1/2},\,\,\,
\eta:=-\frac{\ln\ln\tau^{-1}}{\ln\tau}, 
\end{gathered} 
\end{equation}
and 
\begin{equation}\label{c6_c8}
c_6:=\frac{8}{\pi c_v\mu_o^2}\qquad\text{and}\qquad c_8:=\Big(\frac{8}{c_v\mu_o^2}\Big)^{1/2}; 
\end{equation}

\item[(b3)] if $2n_o+d\ge10,$ then $\mu_o>0$ and for sufficiently small and positive $\mu-\mu_o,$
$$
(- e(\mu))^{1/2}= c_{10}\tau+ \sum\limits_{n+m\ge1,n,m\ge0} c_{10,nm} \tau^{n+1}\theta^m,   
$$
where $\{c_{10,nm}\}$  are some real coefficients and 
\begin{equation}\label{c10}
c_{10}:=(\mu_o^2\hat c_v)^{-1/2}.
\end{equation}
\end{itemize}
\end{itemize}
\end{theorem}

%We conclude this section wih the expansion of $e(\mu)$ for $\mu-\mu^o$ is small enough. In this case the function $q\mapsto  4d^2 - \dispersion(q)$ has a unique non-degenerated minimum at $q=\vec\pi=(\pi,\ldots,\pi)\in\T^d.$ 

\begin{theorem}[\textbf{Expansions of $e(\mu)$ at $\mu=-\mu^o$}] \label{teo:exp_e_atmax}
Let $n^o\ge0$ be such that
\begin{equation*} 
|v(\vec\pi)|^2=D^2|v(\vec\pi)|^2 = \ldots=D^{2n^o-2}|v(\vec\pi)|^2 = 0,\qquad D^{2n^o}|v(\vec\pi)|^2\ne0.
\end{equation*}
Then 
$$
\begin{aligned} 
C_v:=&\frac{2^{2n^o+d-1}}{(8d)^{n^o+d/2}\,(2n^o)!}\, 
 \int_{\S^{d-1}} D^{2n^o}|v(\vec\pi)|^2[w,\ldots,w]\,d\cH^{d-1}(w)>0
\end{aligned}
$$
and 
$$
\hat C_v:= \int_{\T^d} \frac{v(q)^2dq}{(4d^2 - \dispersion(q))^2}\in(0,+\infty].
$$
Moreover, if $2n_o+d\ge3,$ then the equation $\hamiltonian_{-\mu^o} f=4d^2f$ has a non-zero solution 
$$
f(p) = \frac{v(p)}{4d^2 - \dispersion(p)} 
$$
and $f\notin L^2(\T^d)$ for $2n^o+d\in \{3,4\},$ and 
$f\in L^2(\T^d)$ for $2n^o+d\ge5.$

Finally, for $\mu<-\mu^o$ let $e(\mu)>4d^2$ be the eigenvalue of $\hamiltonian_\mu.$ 

\begin{itemize}
\item[(a)] Suppose that $d$ is odd:
\begin{itemize}
\item[(a1)] if $2n^o+d=1,$ then $\mu^o=0$ and for sufficiently small and negative $\mu,$
$$
(e(\mu) - 4d^2)^{1/2}= -C_1\mu + \sum\limits_{n\ge1} C_{1,n} \mu^{n+1}, 
$$
where $\{C_{1,n}\}$  are some real coefficients and $C_1:=\pi C_v;$

\item[(a2)] if $2n^o+d=3,$ then $\mu^o>0$ and for sufficiently small and positive $\mu+\mu^o,$
$$
(e(\mu) - 4d^2)^{1/2}= 
C_3 (\mu+\mu^o) +\sum\limits_{n\ge1} C_{3,n} (\mu+\mu^o)^{n+1},
$$
where $\{C_{3,n}\}$ and $\{C_{7,n}\}$ are some real coefficients and  $C_3:=(\pi C_v{\mu^o}^2)^{-1};$

\item[(a3)] if $2n^o+d\ge5,$ then $\mu^o>0$ and for sufficiently small and positive $\mu+\mu^o,$
$$
(e(\mu) - 4d^2)^{1/2}= C_5(\mu+\mu^o)^{1/2} + \sum\limits_{n\ge1} C_{5,n} (\mu+\mu^o)^{(n+1)/2},  
$$
where $\{C_{5,n}\}$  are some real coefficients  and $C_5:=(\hat C_v {\mu^o}^2)^{-1/2}.$

\end{itemize}

\item[(b)] Suppose that $d$ is even:
\begin{itemize}
\item[(b1)] if $2n_o+d=2,$ then $\mu_o=0$ and for sufficiently small and negative $\mu,$
$$
e(\mu) - 4d^2=  
ce^{\frac{1}{C_2\mu}} +\sum\limits_{n+m\ge1,n,m\ge0} C_{2,nm} \mu^{n+1}\big(-\frac1\mu\,e^{\frac{1}{C_2\mu}}\big)^{m+1}, 
$$
where $\{C_{2,nm}\}$ are some real coefficients and $c>0$ and $C_2:=C_v;$

\item[(b2)] if $2n_o+d=4,$ then $\mu^o>0$ and for sufficiently small and positive $\mu+\mu^o,$
$$
C_4 \mu\sigma+ \sum\limits_{n+m+k\ge1,n,m,k\ge0} C_{4,nmk} \tau^{n+1} \sigma^{m+1}\eta^k,
$$
where $\{C_{4,nm}\}$ are some real coefficients, $C_4:=(C_v{\mu^o}^2)^{-1},$
$$
\tau:=\mu+\mu^o,\qquad 
\sigma :=-\frac{1}{\ln\tau},\qquad 
\eta:=-\frac{\ln\ln\tau^{-1}}{\ln\tau}; 
$$

\item[(b3)] if $2n_o+d\ge6,$ then $\mu^o>0$ and for sufficiently small and positive $\mu+\mu^o,$
$$
e(\mu)-4d^2 = C_6(\mu+\mu^o)+ \sum\limits_{n+m\ge1,n,m\ge0} C_{6,nm} (\mu+\mu^o)^{n+1}(-(\mu+\mu^o)\ln (\mu+\mu^o))^m,   
$$
where $\{C_{6,nm}\}$  are some real coefficients and $C_6:=(\hat C_v{\mu^o}^2)^{-1}.$
\end{itemize}
\end{itemize}

\end{theorem}

We recall that in the literature the non-zero solutions of equations $\hamiltonian_\mu f=0$ and $\hamiltonian_\mu f= 4d^2 f$ not belonging to $L^2(\T^d)$ are called the {\it resonance states} \cite{ALM:04:Puan,ALMM:2006}.

\begin{remark}\label{rem:comments}
Few comments on the main results are in order.

\begin{itemize}
\item[1.]  The assertions of Theorem \ref{teo:existence_disc_spec} hold in fact for any $\hat v\in\ell^2(\Z^d)$  (see Remark \ref{rem:relaxed_assump});

\item[2.] Similar expansions of $e(\mu)$ in Theorems \ref{teo:exp_e_at0} and \ref{teo:exp_e_atmax}  at $\mu=\mu_o$ and $\mu=-\mu^o,$ respectively, still hold for any exponentially decaying $\hat v:\Z^d\to\C$ (see Remark \ref{rem:v_complex_case});

\item[3.] If $\hat v$ decays at most polynomially at infinity, i.e. $\hat v(x) = O(|x|^{-\alpha})$ for some $\alpha>0,$  then instead of the expansions in Theorem \ref{teo:exp_e_at0} and \ref{teo:exp_e_atmax} we obtain only  asymptotics of $e(\mu)$ (see Remark \ref{rem:e_mu_asymp}).
\end{itemize}

\end{remark} 

\section{Proof of main results}\label{sec:proof_mains}

In this section we prove the main results. 

\begin{lemma}\label{lem:determinantga_kelish}
$z\in\C\setminus[0,4d^2]$ is an eigenvalue of $\hamiltonian_\mu$ if and only if 
$\Delta(\mu;z)=0,$ where
$$
\Delta(\mu;z):=1 - \mu\int_{\T^2} \frac{|v(q)|^2dq}{\dispersion(q) - z}.
$$
\end{lemma}

\begin{proof}
Suppose that $z\in \C\setminus[0,4d^2]$  is an eigenvalue of $\hamiltonian_\mu$ with an associated eigenfunction $0\ne f\in L^2(\T^d).$
Then from $\hamiltonian_\mu f=zf$ it follows that 
\begin{equation} \label{f_p}
f(p) = \frac{Cv(p)}{\dispersion(p) -z},
\end{equation}
where 
\begin{equation}\label{C}
C:=\mu \int_{\T^d} \overline{v(q)}f(q)dq.
\end{equation}
Inserting the representation \eqref{f_p} of $f$ in \eqref{C} we get
\begin{equation}\label{delta} 
C= C\mu \int_{\T^d} \frac{|v(q)|^2dq}{\dispersion(q) - z}. 
\end{equation}
Since $C\ne0$ (otherwise $f=0$ by \eqref{f_p}), 
from \eqref{delta}  it follows that $\Delta(\mu;z)=0.$

Conversely, suppose that $\Delta(\mu;z)=0$ for some $z\in\C\setminus[0,4d^2],$ 
and set 
$$
f(p):=\frac{v(p)}{\dispersion(p) - z}.
$$
Then $f\in L^2(\T^d)\setminus\{0\}$ and
$$
(\hamiltonian_\mu - z)f(p) = 1 - \mu\int_{\T^d} \frac{|v(q)|^2dq}{\dispersion(q) - z} = \Delta(\mu;z) =0,
$$
i.e.  $z$ is an eigenvalue of $\hamiltonian_\mu$ with associated eigenvector $f.$
\end{proof}

Notice that 
$$
z\in\C\setminus[0,4d^2] \mapsto \Delta(\mu;z) 
$$ 
is analytic for any $\mu\in\R$ and and 
$$
\frac{\p }{\p z}\,\Delta(\mu;z) = -\mu \int_{\T^d}\frac{|v(q)|^2dq}{(\dispersion(q) - z)^2}.
$$
Thus, for any $\mu>0$ resp. for any $\mu<0,$ the function $\Delta(\mu;\cdot)$ is strictly decreasing resp. strictly increasing in $\R\setminus [0,4d^2].$ Moreover,
\begin{equation}\label{Delta_xosdda}
\lim\limits_{z\to\pm\infty} \Delta(\mu;z) =1.  
\end{equation}

\begin{proof}[Proof of Theorem \ref{teo:existence_disc_spec}]
By the definition of $\mu_o,$ for any $\mu<-\mu^o:$ 
$$
\lim\limits_{z\nearrow-\mu^o} \Delta(\mu;z) = 1 + \frac{\mu}{\mu^o}<0,\qquad  \lim\limits_{z\to+\infty} \Delta(\mu;z) =1.
$$
Since $\Delta(\mu;z)>1$ for $z<0$ and $\mu>-\mu^o,$ 
in view of  the strict monotonicity $\Delta(\mu;\cdot)$ in $(4d^2,\infty),$ for any $\mu<-\mu^o$
there exists a unique $e(\mu)\in(4d^2,+\infty)$ such that $\Delta(\mu;e(\mu))=0.$
Analogously, for any $\mu>\mu_o$ there exists a unique $e(\mu)\in(-\infty,0)$ such that $\Delta(\mu;e(\mu))=0.$
By the Implicit Function Theorem  the function
$\mu\in\R\setminus[-\mu^o,\mu_o]\mapsto  e(\mu)$ is real-analytic.
Moreover, computing the derivatives of the implicit function $e(\mu)$ we find:
\begin{equation}\label{e_mu_der}
e'(\mu) = - \frac{1}{\mu}\,\int_{\T^d} \frac{|v(q)|^2dq}{\dispersion(q) - e(\mu)}\,\Big( \int_{\T^d} \frac{|v(q)|^2dq}{(\dispersion(q) - e(\mu))^2}\Big)^{-1},\qquad \mu\ne0, 
\end{equation}
thus, using $\mu(\dispersion(q) - e(\mu))>0$ we get $e'(\mu)<0,$ i.e. $e(\cdot)$ is strictly decreasing in $\R\setminus\{0\}.$
Differentiating \eqref{e_mu_der} one more time we get
$$
e''(\mu) =  \frac{2e'(\mu)}{\mu}\,\left(1-\mu e'(\mu)\,\int_{\T^d} \frac{|v(q)|^2dq}{(\dispersion(q) - e(\mu))^3}\,\left( \int_{\T^d} \frac{|v(q)|^2dq}{(\dispersion(q) - e(\mu))^2}\right)^{-1}\right).
$$
Therefore, $e''(\mu)>0$ (i.e. $e(\cdot)$ is strictly convex) for $\mu<0$ and $e''(\mu)<0$ (i.e. $e(\cdot)$ is strictly concave) for $\mu>0.$  

To prove \eqref{asymptotics_e_infty},  
first we let $\mu\to\pm\infty$ in 
\begin{equation}\label{adjkkjd}
1= \mu\int_{\T^d} \frac{|v(q)|^2dq}{\dispersion(q) - e(\mu)} 
\end{equation}
and find  $\lim\limits_{\mu\to\pm\infty} e(\mu)=\mp\infty.$ In particular, 
if $|\mu|$ is sufficiently large, $|\frac{\dispersion(q)}{e(\mu)}|<\frac12$ and hence, 
by \eqref{adjkkjd} and the Dominated Convergence Theorem,
$$
\lim\limits_{\mu\to\pm\infty} \frac{e(\mu)}{\mu} = -  \lim\limits_{\mu\to\pm\infty} \int_{\T^d}\frac{|v(q)|^2dq}{1- \frac{\dispersion(q)}{e(\mu)}} = -\int_{\T^d} |v(q)|^2dq.
$$
\end{proof}

\begin{remark}\label{rem:relaxed_assump}
According to the proofs of Lemma \ref{lem:determinantga_kelish} and Theorem \ref{teo:existence_disc_spec}, their assertions still holds for any $v\in \ell^2(\Z^d).$ 
\end{remark}

\begin{proof}[Proof of Theorem \ref{teo:exp_e_at0}]
Since 
\begin{equation}\label{v_kv_ning_ifodasi}
|v(p)|^2= (2\pi)^{-d}\Big(\sum\limits_{x\in\Z^d} \hat v(x)\cos p\cdot x\Big)^2+(2\pi)^{-d}\Big(\sum\limits_{x\in\Z^d} \hat v(x)\sin p\cdot x\Big)^2, 
\end{equation}
the function $p\in\T^d\mapsto |v(p)|^2$ is nonnegative even real-analytic function. 
Notice also that if $n_o\ge1,$ then by the nonnegativity of $|v|^2,$ $p=0$ is a global minimum for $|v|^2.$ Therefore, the tensor
$
D^{2n_o}|v(0)|^2 
$ is positively definite and
\begin{equation}\label{c0}
c_v:=\frac{2^{2n_o+d}}{(2n_o)!} \int_{\S^{d-1}} D^{2n_o} |v(0)|^2[w,\ldots,w]d\cH^{d-1}>0.
\end{equation}
Note  that    
$$
\hat c_v = \low_{|v|^2}'(0)= \int_{\T^d}\frac{|v(q)|^2dq}{\dispersion(q)^2},
$$
where $\low_f$ is defined in \eqref{def_l_f}. By Proposition \ref{prop:l_f_asymp}, $f(p)=\frac{v(p)}{\dispersion(p)}\in L^2(\T^d)$ if and only if $2n_o+d\ge9.$ Moreover, by definition, $\mu_o>0$ and  $\Delta(\mu_o;0)=0$  for $2n_o+d\ge5,$ and hence,  as in the proof of Lemma \ref{lem:determinantga_kelish} for such $d$ one can show that $\hamiltonian_{\mu_o}f=0.$ 

In view of the strict monotonicity and \eqref{asymptotics_e_chegara} there exists a unique $\mu_1>0$ such that $e(\mu)\in(-\frac{1}{128},0)$ for any $\mu\in(0,\mu_1).$ Since 
\begin{equation}\label{eq_wrt_tau0}
\mu = (\low_{|v|^2}(e(\mu)))^{-1}, 
\end{equation}
we can use Proposition \ref{prop:l_f_asymp} with $f=|v|^2$ and $e:=e(\mu),$
to find the expansions of the inverse function $\mu:=\mu(e).$ Then applying the appropriate versions of the Implicit Function Theorem in analytical case we get the expansions of $e=e(\mu).$
Notice that from \eqref{l_f_asymp_d_odd} and \eqref{l_f_asymp_d_even} as well as \eqref{eq_wrt_tau}  it follows that $\mu_o=0$ for $2n_o+d \le  4$ and $\mu_o=\Big(\int_{\T^d} \frac{|v(q)|^2dq}{\dispersion(q)}\Big)^{-1}>0$ for $2n_o+d\ge5.$ 

(a) Suppose that $d$ is odd. In view of the expansions \eqref{l_f_asymp_d_odd} of $\low_f$, in this case, \eqref{eq_wrt_tau0} is reduced to the inverting the equation 
\begin{equation}\label{eq_wrt_tau}
\mu = g(\alpha),
\end{equation}
where $\alpha:=(-e)^{1/4}$ and $g$ is an analytic function around $\alpha=0.$  
%By \eqref{l_f_asymp_d_odd}, 

{\it Case $2n_o+d=1.$}  In this case by \eqref{l_f_asymp_d_odd},  
$$
g(\alpha) := \frac{\alpha^3}{c_1^3 + \sum\limits_{n\ge1} a_n \alpha^n}, 
$$
where $\{a_n \}\subset\R$ and  $c_1$ is given by \eqref{c1_c3},
and \eqref{eq_wrt_tau} is equivalently represented as  
\begin{equation}\label{inv_eq_d1}
\alpha  = \mu\Big(c_1^3 + \sum\limits_{n\ge1} a_n \alpha^{n}\Big)^{1/3},
\end{equation}
where $\mu=\mu^{1/3}.$ Now setting 
\begin{equation}\label{oshk_d1}
\alpha= \mu(c_1 + u), 
\end{equation}
and using the Taylor series of $(c_1^3+x)^{1/3},$
for $\mu$ and $u$ sufficiently small we rewrite \eqref{inv_eq_d1} as 
\begin{equation}\label{oshkormas_d1}
F(u,\mu):= u - \sum\limits_{n\ge1} \tilde a_n  \mu^n(c_1 + u)^n =0, 
\end{equation}
where $F(\cdot,\cdot)$ is analytic at $(u,\mu)=(0,0),$  
$F(0,0)=0$ and $F_u(0,0) =1.$ Hence, by the Implicit Function Theorem, there exists $\gamma_1>0$ such that for $|\mu|<\gamma_1,$ \eqref{oshkormas_d1} has a unique real-analytic  solution $u=u(\mu)$ which can be represented as an absolutely convergent series  $u=\sum\limits_{n\ge1} b_n\mu^n.$
Putting this in \eqref{oshk_d1} and recalling the definitions of $\alpha$ and $\mu$ we get the expansion of $(-e(\mu))^{1/4}$ for $\mu>0$ small.  

{\it Case $2n_o+d=3.$} By \eqref{l_f_asymp_d_odd},
\begin{equation}\label{inv_eq_d3}
g(\alpha)=\alpha \Big(c_3 +  \sum\limits_{n\ge1} a_n \alpha^n\Big)^{-1},  
\end{equation}
where $\{a_n\}\subset\R$ and $c_3$ is given by \eqref{c1_c3},
and hence\eqref{eq_wrt_tau} is represented as  
$$
\alpha = \mu \Big(c_3 +  \sum\limits_{n\ge1} a_n \alpha^n\Big).
$$
Then setting $\alpha=\mu(c_3+u)$ we rewrite \eqref{inv_eq_d3} in the form \eqref{oshkormas_d1}, and as in the case of $2n_o+d=1,$ we get the expansion of $(-e(\mu))^{1/4}.$

{\it Case $2n_o+d=5.$}  In this case by \eqref{l_f_asymp_d_odd}  
\begin{equation}\label{inv_eq_d5}
g(\alpha)= \Big(\frac{1}{\mu_o} - \frac{\pi c_v\alpha}{8}\Big(1+ \sum\limits_{n\ge1}a_n\alpha^n \Big)\Big)^{-1}, 
\end{equation}
where  $\{a_n\}\subset\R,$ and hence, by \eqref{eq_wrt_tau}, 
\begin{equation}\label{inv_eq_d5p}
\frac{\mu - \mu_o}{\mu\mu_o}= \frac{\pi c_v\alpha}{8}\Big(1+  \sum\limits_{n\ge1}a_n\alpha^n \Big). 
\end{equation}
Note that if $|\mu-\mu_o|<\mu_o,$ then 
\begin{equation}\label{adadd}
\frac{\mu - \mu_o}{\mu\mu_o} = \frac{\mu - \mu_o}{{\mu_o}^2+ \mu_o(\mu - \mu_o)} = \frac{\mu - \mu_o}{{\mu_o}^2} \sum\limits_{n\ge0} \Big(\frac{\mu - \mu_o}{\mu_o}\Big)^n, 
\end{equation}
thus from \eqref{inv_eq_d5p} we get 
$$
\alpha = (\mu - \mu_o)\Big(c_5 + c_5 \sum\limits_{n\ge1} \mu_o^{-n} (\mu - \mu_o)^n\Big)\,\Big(1+  \sum\limits_{n\ge1}a_n\alpha^n \Big)^{-1}. 
$$
and $c_5$ is given by \eqref{c5_c7}.  
Now setting $\alpha=(\mu-\mu_o)\,(c_5+u)$
 for sufficiently small and positive $\mu-\mu_o$ we get
$$
u=\sum\limits_{n,m\ge1} \tilde c_{n,m}(\mu - \mu_o)^n (c_5 + u)^m,
$$
where $\tilde c_{n,m}\subset\R.$ By the Implicit Function Theorem, 
for sufficiently small $\mu-\mu_o$ there exists a unique real-analytic function $u=u(\mu)$ given by the absolutely convergent series 
$u(\mu) = \sum\limits_{n\ge1} b_n(\mu - \mu_o)^n.$ By the definition of $\alpha, $ this implies  the expansion of $(-e(\mu))^{1/4}.$

{\it Case $2n_o+d=7.$}  As the previous case, by \eqref{l_f_asymp_d_odd} and \eqref{adadd}, the equation \eqref{eq_wrt_tau} is represented as  
\begin{equation}\label{inv_eq_d7}
(\mu - \mu_o)\Big(c_7^3 + c_7^3\sum\limits_{n\ge1} \mu_o^{-n}(\mu - \mu_o)^n\Big) = \alpha^{3} \Big(1 + \sum\limits_{n\ge1}a_n\alpha^n \Big), 
\end{equation}
where  $\{a_n\}\subset\R$   and $c_7$ is given by \eqref{c5_c7}. When $\mu-\mu_o>0$ is small enough, by the Taylor series of $(1+x)^{\pm1/3}$ at $x=0,$ \eqref{inv_eq_d7} is equivalently rewritten as 
\begin{equation}\label{inv_eq_d7p}
\alpha = (\mu - \mu_o)^{1/3}\Big(c_7 + \sum\limits_{n\ge1} \tilde c_n(\mu - \mu_o)^n\Big) \Big(1 + \sum\limits_{n\ge1}\tilde a_n\alpha^n \Big), 
\end{equation}
Thus, for   $\mu=(\mu-\mu_o)^{1/3},$ setting  $\alpha=\mu\,(c_7+u)$ in \eqref{inv_eq_d7p}, for sufficiently small and positive $\mu$ we get
$$
u=\sum\limits_{n,m\ge1} \tilde c_{n,m}\mu^n(c_7+u)^m.
$$
By the Implicit Function Theorem, this equation has a unique real-analytic solution $u=u(\mu)$ given by the absolutely convergent series $u=\sum\limits_{n\ge1} b_n\mu^n.$ 
This, definitions of $\alpha$ and $\mu$ imply the expansion of $(-e(\mu))^{1/4}.$

{\it Case $2n_o+d=9.$}  In this case by \eqref{l_f_asymp_d_odd}  and \eqref{adadd} 
\begin{equation}\label{inv_eq_d9}
(\mu - \mu_o)\Big(c_9^4 + c_9^4\sum\limits_{n\ge1} \mu_o^{-n}(\mu - \mu_o)^n\Big) = \alpha^4 \Big(1 + \sum\limits_{n\ge1}a_n\alpha^n \Big), 
\end{equation} 
where $\{a_n\}\subset\R$ and $c_9$ is given by \eqref{c9}. Thus, for sufficiently small and positive $\mu-\mu_o$ using the Taylor series of $(1+x)^{\pm1/4}$ at $x=0,$ this equation can also be represented as 
\begin{equation}\label{inv_eq_d9p}
\alpha = \mu \Big(c_9 + \sum\limits_{n\ge1} \tilde b_n\mu^{4n}\Big) \Big(1 + \sum\limits_{n\ge1}\tilde a_n\alpha^n\Big),
\end{equation} 
where $\mu:= (\mu - \mu_o)^{1/4}.$
Now  setting $\alpha= \mu(c_9+ u)$ in \eqref{inv_eq_d9} we get
$$
u=\sum\limits_{n,m\ge1} \tilde c_{n,m}\mu^n(c_9+u)^m,
$$
and the expansion of $(-e(\mu))^{1/4}$ follows as in the case of $2n_o+d=7.$

(b) Suppose that $d$ is even. In view of the expansion \eqref{l_f_asymp_d_odd} of $\low_f$, in this case, \eqref{eq_wrt_tau0} is reduced to the inverting the equation 
\begin{equation}\label{eq_wrt_tau1}
\mu = \frac{\alpha^l}{g(\alpha) +  h(\alpha) \ln\alpha },
\end{equation}
where $\alpha:=(-e)^{1/2},$ $l\in\N_0,$ and  $g$ and $h$ are analytic around $\alpha=0.$ Presence of $\ln\alpha$ implies that unlike the case of $2n_o+d$ odd, $\alpha$ is not necessarily analytic with respect to $\mu.$ Therefore, we need to introduce new variables dependent on $\ln\mu$ to reduce the problem to the Implicit Function Theorem.

{\it Case $2n_o+d=2.$} By \eqref{l_f_asymp_d_even}, in this case 
$$
l=1,\qquad g(\alpha)=c_2 + \sum\limits_{n\ge1} a_n\alpha^n,\qquad h(\alpha) =\sum\limits_{n\ge1} b_n\alpha^{2n}.
$$
Hence, setting 
\begin{equation}\label{lafa2}
\alpha=\mu(c_2 + u) 
\end{equation}
and $\tau = -\mu\ln\mu$ we represent  \eqref{eq_wrt_tau1} as 
\begin{align}
F(u,\mu,\tau):=u - \sum\limits_{n\ge1} a^n\mu^n(c_2+u)^n & + \ln(c_2+u)\sum\limits_{n\ge1} b^n\mu^n(c_2+u)^n\no \\
&-\tau\sum\limits_{n\ge1} b^n\mu^{n-1}(c_2+u)^n=0, \label{oshkormas_d2}
\end{align}
where $F$ is analytic around $(0,0,0),$ $F(0,0,0)=0,$ $F_u(0,0,0)=1.$ Hence, by the Implicit Function Theorem, there exists a unique real-analytic function $u=u(\mu,\tau)$ given by the convergent series $u(\mu,\tau) = \sum\limits_{n+m\ge1,n,m\ge0} \tilde c_{n,m}\mu^n\tau^m$ for sufficiently small $|\mu|$ and $|\tau|,$ which satisfies $F(u(\mu,\tau),\mu,\tau)\equiv0.$ 
Inserting $u$ in \eqref{lafa2} we get the expansion of $\alpha=(-e)^{1/2}.$ 

{\it Case $2n_o+d=4.$}  In this case, by \eqref{l_f_asymp_d_even},
$$
l=0,\qquad g(\alpha)=\sum\limits_{n\ge0} a_n\alpha_n,\qquad h(\alpha) = -c_4 + \sum\limits_{n\ge1} b_n\alpha^{2n}.
$$
Letting $\alpha= e^{-\frac{1}{c_4\mu}}(c+u),$ where $c= e^{a_0/c_4}>0,$   
we represent \eqref{eq_wrt_tau1} as 
\begin{align}
\ln(c+u) - b_0 = &\frac{1}{\mu}\,e^{-\frac{1}{c_4\mu}} \sum\limits_{n\ge1} a^ne^{-\frac{n-1}{c_4\mu}} (c+u)^n\no\\
 +  &\ln(c+u)\sum\limits_{n\ge1} b^ne^{-\frac{n}{c_4\mu}}(c+u)^n 
-\sum\limits_{n\ge1} a^ne^{-\frac{n}{c_4\mu}}(c+u)^n=0. \label{oshkormas_d4}
\end{align} 
Writing $\tau:=\frac{1}{\mu}\,e^{-\frac{1}{c_4\mu}}$ so that $e^{-\frac{1}{c_4\mu}} = \mu\tau,$ \eqref{oshkormas_d4} is represented as 
\begin{align*}
F(u,\mu,\tau):= &\ln(c+u) - b_0 -\mu\sum\limits_{n\ge1} a^n\mu^{n-1}\tau^{n-1} (c+u)^n\\ 
&-\ln(c+u)\sum\limits_{n\ge1} b^n\mu^n\tau^n(c+u)^n 
 +\sum\limits_{n\ge1} a^n\mu^n\tau^n(c+u)^n=0,
\end{align*}
where $F$ is analytic around $(0,0,0),$ $F(0,0,0)=0,$ and $F_u(0,0,0)=\frac1c>0.$
Thus, by the Implicit Function Theorem, for $|\mu|,$ $|\tau|$ and $|u|$ small there exists a unique real analytic function $u=u(\mu,\tau)$ given by the convergent series 
$u=\sum\limits_{n+m\ge1,n,m\ge0} \tilde c_{n,m}\mu^n\tau^m$ such that $F(u(\mu,\tau),\mu,\tau)\equiv0.$  Since $\tau = \frac1\mu e^{-\frac{1}{c_4\mu}},$ this implies 
$$
\alpha= e^{-\frac{1}{c_4\mu}}(c+u) = ce^{-\frac{1}{c_4\mu}} + \sum\limits_{n+m\ge1,n,m\ge0} \tilde c_{n,m}\mu^{n+1}\Big(\frac1\mu e^{-\frac{1}{c_4\mu}}\Big)^{m+1}.
$$

{\it Case $2n_o+d=6.$}  In this case, by \eqref{l_f_asymp_d_even} and the definition \eqref{c6_c8} of $c_6,$,
$$
l=0,\qquad g(\alpha)=\frac{1}{\mu_o} -\frac{1}{c_6\mu_o^2} \Big(\alpha + \sum\limits_{n\ge2} a_n\alpha^n\Big),\qquad h(\alpha) = \frac{1}{c_6\mu_o^2} \sum\limits_{n\ge1} b_n\alpha^{2n},
$$
and hence, \eqref{eq_wrt_tau1} is represented as 
$$
\frac{1}{\mu} -\frac{1}{\mu_o} = \frac{1}{c_6\mu_o^2}\Big(\alpha  +\sum\limits_{n\ge2} a_n\alpha^n  + \ln\alpha  \sum\limits_{n\ge1} b_n\alpha^{2n}\Big), 
$$
or equivalently, by \eqref{adadd},
\begin{equation}\label{inv_eq_d6}
\alpha = c_6(\mu - \mu_o) \sum\limits_{n\ge0} \Big(\frac{\mu - \mu_o}{\mu_o}\Big)^n - \sum\limits_{n\ge2} a_n\alpha^n -\ln\alpha  \sum\limits_{n\ge1} b_n\alpha^{2n}. 
\end{equation}
Recalling the definitions of $\tau$ and $\theta$ in \eqref{new_variables},
setting $\alpha= \tau^2\,(c_6+u),$  we represent \eqref{inv_eq_d6} as 
\begin{align*}
F(u,\tau, \theta):=  u &-c_6\sum\limits_{n\ge1}\frac{\tau^{2n}}{\mu_o^n}- \sum\limits_{n\ge2} a_n\tau^{2n-2}(c_6+u)^n\no \\
&- \ln(c_6+u)  \sum\limits_{n\ge1} b_n\tau^{4n}(c_6+u)^{2n} -   
\theta \sum\limits_{n\ge1} b_n\tau^{4n-4}(c_6+u)^{2n}=0,
\end{align*} 
where $F$ is real-analytic around $(0,0,0),$ $F(0,0,0)=0$ and $F_u(0,0,0)=1,$ and $F$ is even in $\tau.$ 
Thus, by the Implicit Function Theorem, for $|u|,$ $|\tau|$ and $|\theta|$ small there exists a unique real analytic function $u=u(\tau,\theta),$ even in $\tau,$ given by the convergent series $u=\sum\limits_{n+m\ge1,n,m\ge0} \tilde c_{n,m}\tau^{2n}\theta^m$ such that $F(u(\tau,\theta),\tau,\theta)\equiv0.$ Thus, 
$$
\alpha= \tau^2\,(c_6+u) = c_6\sigma + \sum\limits_{n+m\ge1,n,m\ge0}  \tilde c_{n,m}\tau^{2n+2}\theta^m.
$$

{\it Case $2n_o+d=8.$} By \eqref{l_f_asymp_d_even} and the definition \eqref{c6_c8} of $c_8,$
$$
l=0,\qquad g(\alpha)= \frac{1}{\mu_o^2c_8^2}\sum\limits_{n\ge2} a_n\alpha^n,\qquad h(\alpha) =\frac{1}{\mu_o^2c_8^2}\Big(\alpha^2 + \sum\limits_{n\ge2} b_n\alpha^{2n}\Big),
$$
thus, as in the case of $2n_o+d=6,$ \eqref{eq_wrt_tau1} is represented as 
\begin{equation}\label{inv_eq_d8}
c_8^2(\mu - \mu_o) \sum\limits_{n\ge0} \Big(\frac{\mu - \mu_o}{\mu_o}\Big)^n =   \alpha^2\ln\alpha  +\ln\alpha \sum\limits_{n\ge2} b_n\alpha^{2n}+\sum\limits_{n\ge2} a_n\alpha^n.  
\end{equation}
For $\tau,$ $\sigma$ and $\eta$ given in \eqref{new_variables} set $\alpha = \tau\sigma(c_8+u)$ and 
represent \eqref{inv_eq_d8} as 
\begin{align*}
2c_8u+  u^2 = & c_8^2\sum\limits_{n\ge1} \frac{\tau^{2n}}{\mu_o^n} + \sum\limits_{n\ge2}a_n \tau^{n-1} \sigma^{n+1}(c_8+u)^{n+2} -\sum\limits_{n\ge2}b_n (\tau\sigma)^{2n-2}(c_8+u)^{2n+2}\\
& + \Big( \sigma^2\ln(c_8+u) -\frac{\eta}{2}  \Big)\Big((c_8+u)^2+\sum\limits_{n\ge2}b_n (\tau\sigma)^{2n-2}(c_8+u)^{2n+2}\Big).
\end{align*}
This equation is represented as 
$
F(u,\tau,\sigma,\eta)=0,
$ 
where $F$ is real-analytic in a neighborhood of $(0,0,0,0),$ $F(0,0,0,0)=0$ and $F_u(0,0,0,0) =2c_8>0.$ Hence, for $|u|,$ $|\tau|,$ $|\sigma|$ and $|\eta|$ small, by the Implicit Function Theorem, there exists a unique real-analytic function $u=u(\tau,\sigma,\eta)$ given by the convergent series $u=\sum\limits_{n+m+k\ge1,n,m,k\ge0} \tilde c_{n,m,k}\tau^n\sigma^m\mu^k$ such that $F(u(\tau,\sigma,\eta),\tau,\sigma,\eta)\equiv 0.$ Thus,
$$
\alpha= \tau\sigma(c_8+u) =c_8\tau\sigma + \sum\limits_{n+m+k\ge1,n,m,k\ge0} \tilde c_{n,m,k}\tau^{n+1}\sigma^{m+1}\eta^k.
$$

{\it Case $2n_o+d\ge10.$} By \eqref{l_f_asymp_d_even} and the definition \eqref{c10} of $c_{10},$ 
$$
l=0,\qquad g(\alpha)= \frac{1}{\mu_o} + \hat c_v \alpha^2 + \sum\limits_{n\ge2} a_n\alpha^{n+2},\qquad h(\alpha) =\sum\limits_{n\ge2} b_n\alpha^{2n},
$$
and as in the case of $2n_o+d=6,$ \eqref{eq_wrt_tau1} is represented as 
\begin{equation}\label{inv_eq_d10}
\frac{\mu - \mu_o}{\mu_o^2} \sum\limits_{n\ge0} \Big(\frac{\mu - \mu_o}{\mu_o}\Big)^n =   \hat c_v \alpha^2 +  \sum\limits_{n\ge2} a_n\alpha^{n+2}+ \ln\alpha \sum\limits_{n\ge2} b_n\alpha^{2n}. 
\end{equation}
Recalling the definitions of $\tau$ and $\theta$ in \eqref{new_variables}, we set $\alpha=\tau(c_{10} +u).$  Then \eqref{inv_eq_d10} is represented as 
\begin{align*}
F(u,\tau,\theta):=  & 2c_{10}u  +  u^2 - c_{10}^2\sum\limits_{n\ge1} \frac{\tau^{2n}}{\mu_o^n}
 + \sum\limits_{n\ge2}a_n \tau^n (c_{10}+u)^{n+2}\\
 &- \theta\sum\limits_{n\ge2}b_n \tau^{2n-4}(c_8+u)^{2n} 
 + \ln(c_{10}+u)\sum\limits_{n\ge2}b_n \tau^{2n-2}(c_8+u)^{2n}=0,
\end{align*}
where $F$ is analytic at $(0,0,0),$ $F(0,0,0)=0$ and $F_u(0,0,0)=2c_{10}>0.$ Thus, by the Implicit Function Theorem, for $|u|,$ $|\tau|$ and $|\theta|$  small there exists a unique real-analytic function $u=u(\tau,\theta)$ given by the convergent series $u=\sum\limits_{n+m\ge1,n,m\ge0}\tilde c_{n,m}\tau^n\theta^n$ such that 
$F(u(\tau,\theta),\tau,\theta)\equiv0.$ Then 
$$
\alpha= \mu(c_{10} +u) = c_{10}\mu + \sum\limits_{n+m\ge1,n,m\ge0}\tilde c_{n,m}\mu^{n+1}\theta^n.
$$
Theorem is proved.
\end{proof}

\begin{proof}[Proof of Theorem \ref{teo:exp_e_atmax}]
From \eqref{v_kv_ning_ifodasi} it follows that the map $p\in\T^d\mapsto f(\vec\pi+p)$ is even. 
Now the  expansions of $e(\mu)$ at $\mu=-\mu^o$ can be proven along the same lines of Theorem \ref{teo:exp_e_at0} using Proposition \ref{prop:U_f_asymp} with $f=|v|^2$,
therefore we drop its proof.  
\end{proof}

\begin{remark}\label{rem:v_complex_case}
Let $\hat v:\Z^d\to\C$ satisfy \eqref{hat_v_ga_shart}. Since $\dispersion(\cdot)$ is even, 
$$
\int_{\T^d} \frac{|v(p)|^2dp}{\dispersion(p) -z} =\frac{1}{(2\pi)^d} \, \int_{\T^d} \frac{f(p)dp}{\dispersion(p) -z},
$$
where  
$$
\begin{aligned}
f(p):=\Big(\sum\limits_{x\in\Z^d} \hat v_1(x) \cos p\cdot x\Big)^2& + \Big(\sum\limits_{x\in\Z^d} \hat v_2(x) \cos p\cdot x\Big)^2\\ 
&+ \Big(\sum\limits_{x\in\Z^d} \hat v_1(x) \sin p\cdot x\Big)^2 + \Big(\sum\limits_{x\in\Z^d} \hat v_2(x) \sin p\cdot x\Big)^2
\end{aligned}
$$
and $\hat v= \hat v_1+ i\hat v_2$ for some $\hat v_1,\hat v_2:\Z^d \to\R.$ By Lemma \ref{lem:determinantga_kelish}, the unique eigenvalue $e(\mu)$ of $\hamiltonian_\mu$
solves 
$$
1-\mu\int_{\T^d} \frac{f(p)dp}{\dispersion(p) - e(\mu)} =0.
$$
Since both $p\in\T^d\mapsto  f(p)$ and  $p\in\T^d\mapsto  f(\vec\pi + p)$  are even analytic functions, we can still apply Propositions \ref{prop:l_f_asymp} and \ref{prop:U_f_asymp} to find the expansions of $z\mapsto \int_{\T^d} \frac{f(p)dp}{\dispersion(p) - z}$ and thus, repeating the same arguments of the proofs of Theorems \ref{teo:exp_e_at0} and \ref{teo:exp_e_atmax} one can obtain the corresponding expansions of $e(\mu).$ 
\end{remark}

\begin{remark}\label{rem:e_mu_asymp}
When 
$$
|\hat v(x)| = O(|x|^{2n_0+d+1})\qquad \text{as $|x|\to \infty$}
$$
for some $n_0\ge1,$ in view of Remark \ref{rem:asymp_low_f}, we need to solve equation \eqref{eq_wrt_tau0} with respect to $\mu$ using only that left-hand side is an asymptotic  sum (not a convergent series). This still can be done using appropriate modification of the  Implicit Function Theorem for differentiable functions. As a result, we obtain only (Taylor-type) asymptotics of $e(\mu).$ 
\end{remark}

\appendix 
\section{Asymptotics of some integrals}\label{sec:asymp_integr}

In this section we study the behaviour of the integral 
\begin{equation}\label{def_l_f}
\low_f(z):=\int_{\T^d} \frac{f(q)dq}{\dispersion(q) -z},\qquad z\in\C\setminus[0,4d^2], 
\end{equation}
as $z\to 0$ and $z\to 4d^2,$ where $f:\T^d\to\R$ is a real-analytic even function on $\T^d.$ 

\begin{proposition}%[\textbf{Asymptotics at $0$}]
\label{prop:l_f_asymp}
Let $f:\T^d\to\R$ be a real-analytic even function 
such that 
\begin{equation}\label{condition_on_f}
f(0) = D^2f(0) = \ldots =D^{2n_o-2}f(0) = 0,\qquad D^{2n_o}(0)\ne0 
\end{equation}
for some $n_o\ge0.$ 
Then:
\begin{itemize}
\item[--] $\low_f$ is continuous at $0$ if and only if $2n_o+d\ge 5;$ 

\item[--] $\low_f$ is continuously differentiable if and only if  $2n+d\ge9,$  
in this case,
\begin{equation}\label{l_fning_hosilasi}
\low_f'(0):=\int_{\T^d} \frac{f(q) dq}{(\dispersion(q))^2}= \lim\limits_{z\searrow0} \int_{\T^d} \frac{f(q)dq}{(\dispersion(q) -z)^2}.   
\end{equation}
\end{itemize}
Moreover, for any $z\in (-\frac{1}{64},0):$
\begin{itemize}
\item[(a)] if $d$ is odd, then   
\begin{equation}\label{l_f_asymp_d_odd}
\low_f(z) = 
\begin{cases}
\frac{\pi}{4(-z)^{3/4}}\, \Big(c_f+ \sum\limits_{n\ge1} a_n^d\,(-z)^{n/4}\Big), & 2n_o+d=1,\\
\frac{\pi}{8(-z)^{1/4}}\, \Big(c_f+ \sum\limits_{n\ge1} a_n^d\,(-z)^{n/4}\Big), & 2n_o+d=3,\\
\low_f(0)-\frac{\pi(-z)^{1/4}}{8}\,\Big(c_f+ \sum\limits_{n\ge1} a_n^d\,(-z)^{n/4}\Big), & 2n_o+d=5,\\
\low_f(0)-\frac{\pi(-z)^{3/4}}{8}\,\Big(c_f+ \sum\limits_{n\ge1} a_n^d\,(-z)^{n/4}\Big), & 2n_o+d=7,\\
\low_f(0) + z\Big(\low_f'(0) + \sum\limits_{n\ge1} a_n^d\,(-z)^{n/4}\Big), & 2n_o+d\ge9,
\end{cases}
\end{equation}

\item[(b)] if $d$ is even, then   
\begin{equation}\label{l_f_asymp_d_even}
\low_f(z) = 
\begin{cases}
\frac{\pi}{8(-z)^{1/2}}\, \Big(c_f+ \sum\limits_{n\ge1} b_n^d\,(-z)^{n/2}\Big) -\frac{1}{16}\,\ln(-z)\sum\limits_{n\ge0} c_n^dz^n , & 2n_o+d=2,\\
 -\frac{1}{16}\,\ln(-z)\Big(c_f + \sum\limits_{n\ge1} c_n^dz^n\Big) +\sum\limits_{n\ge0} b_n^d\,(-z)^{n/2},& 2n_o+d=4,\\
\low_f(0)-\frac{\pi(-z)^{1/2}}{8}\, \Big(c_f+ \sum\limits_{n\ge1} b_n^d\,(-z)^{n/2}\Big) +z\ln(-z)\sum\limits_{n\ge0} c_n^dz^n , & 2n_o+d=6,\\
\low_f(0)- \frac{z}{16}\,\ln(-z)\Big(c_f + \sum\limits_{n\ge1} c_n^dz^n\Big) +\sum\limits_{n\ge2} b_n^d\,(-z)^{n/2},& 2n_o+d=8,\\
\low_f(0) + z\Big(\low_f'(0) + \sum\limits_{n\ge1} b_n^d\,(-z)^{n/2}\Big) + z^2\ln(-z)\sum\limits_{n\ge0} c_n^d\,z^n , & 2n_o+d\ge10,
\end{cases}
\end{equation}
\end{itemize}
where $\{a_n^d\},$ $\{b_n^d\}$ and $\{c_n^d\}$ are some real coefficients,
\begin{equation}\label{c0_and_c1}
c_f:=\frac{2^{2n_o+d}}{(2n_o)!}\, \int_{\S^{d-1}} D^{2n_o}f(0)[w,\ldots,w]d\cH^{d-1};
\end{equation}
and all series in \eqref{l_f_asymp_d_odd} and \eqref{l_f_asymp_d_even} converge absolutely for $z\in W_{1/64}(0)\subset\C.$ 
\end{proposition}

\begin{proof}
Given $\gamma\in(0,\frac{1}{\sqrt2}],$ let $\varphi:B_{\gamma}(0)\subset\R^d\to \varphi(B_{\gamma}(0))\subset\R^d$ be the smooth diffeomorphism 
\begin{equation}\label{morse_map}
\varphi_i(y) = 2 \arcsin y_i,\qquad i=1,\ldots,d.
\end{equation}
Note that
\begin{equation}\label{morse_ch}
\dispersion(\varphi(y)) = \Big(\sum\limits_{i=1}^d (1 - \cos(2\arcsin(y_i)) )\Big)^2 = 4\Big(\sum\limits_{i=1}^d y_i^2\Big)^2 =4 y^4, 
\end{equation}
therefore,  
\begin{equation}\label{lb_for_disp}
\dispersion(q)\ge 4 \gamma^4 \qquad \text{for any $q\in  \T^d\setminus \varphi(B_{\gamma})$}.
\end{equation}
We rewrite $\low_f(z)$ as  
% %
%
\begin{equation}\label{ufa_exp}
\low_f(z): = \int_{\varphi(B_{\gamma}(0))} \frac{f(q)dq}{\dispersion(q) -z} + \int_{\T^d\setminus \varphi(B_{\gamma}(0))} \frac{f(q)dq}{\dispersion(q) -z}:=\low^*(z) +\low^{**}(z). 
\end{equation}
By virtue of \eqref{lb_for_disp}, the series $\sum\limits_{n\ge0} (\frac {z}{\dispersion(q)})^n$ is absolutely convergent  for  $|z|<2\gamma^4$ and $q\in \T^d\setminus \varphi(B_{\gamma}),$ thus, 
\begin{equation}\label{I_2_a_analytic}
\low^{**}(z) = \int_{\T^d\setminus \varphi(B_{\gamma}(0))} \frac{f(q)}{\dispersion(q)}\,\Big(1-\frac{z}{\dispersion(q)}\Big)^{-1}dq  =\sum\limits_{n\ge0}  z^{n} \int_{\T^d\setminus \varphi(B_{\gamma}(0))} \frac{f(q)dq}{(\dispersion(q))^{n+1}}, 
\end{equation}
i.e. $\low^{**}(\cdot)$ is  analytic in $W_{2\gamma^4}(0).$ 
In $\low^*$ making the change of variables $q = \varphi(y)$ and using \eqref{morse_ch} we get  
\begin{equation}\label{I_1_a_aaa}
\low^*(z) = \int_{B_{\gamma}(0)} \frac{f(\varphi(y))\, J(\varphi(y))\,dy}{4y^4 -z}, 
\end{equation}
where $y^4:=(y^2)^2$ with $y^2:=\sum\limits_{i=1}^d y_i^2$, and   
\begin{equation}\label{jacoba}
J(\varphi(y)) = \prod\limits_{i=1}^d \frac{2}{\sqrt{1-y_i^2}} 
\end{equation}
is the Jacobian of $\varphi.$ 
%
%Note that $\low^*$ is real-analytic in $\C\setminus[0,4].$ 
Since $f$ is an even analytic function satisfying \eqref{condition_on_f}, even each coordinate,  from the Taylor series for $f$ it follows that
\begin{equation}\label{f_anal_qator}
f(p) = \sum\limits_{n\ge n_o} \frac{1}{(2n)!}\,D^{2n}f(0)[\underbrace{p,\ldots,p}_{\text{$2n$-times}}],  
\end{equation}
and by the analyticity of $f$ in $B_\pi(0)\subset\R^d,$ the series converges absolutely in $p\in B_\pi(0).$ %By the assumptions on $f,$ $D^{2n}f(0)$ is symmetric and even in each variable.
By the definition of $\varphi,$ $\varphi(rw)\subset B_\pi(0)$ for any $r\in (0,\gamma)$ and $w=(w_1,\ldots,w_d)\in \S^{d-1},$ where $\S^{d-1}$ is the unit sphere in $\R^d.$  Then letting $p=\varphi(rw)$ and using the Taylor series 
\begin{equation}
\varphi_i(rw)= 2rw_i + \frac{r^3w_i^3}{3} +\sum\limits_{n\ge3} \tilde c_n r^{2n-1}w_i^{2n-1} 
\label{taylor_arcsin} 
\end{equation}
of $2\arcsin(\cdot),$  
which is absolutely  convergent for $|rw_i|<1,$
from \eqref{f_anal_qator} we obtain
\begin{equation}\label{f_ays}
f(\varphi(rw))= \sum\limits_{n\ge n_o} \tilde C_n(w)\,r^{2n}, 
\end{equation}
where $\tilde C_n:\S^{d-1}\to \R$ is a homogeneous   polynomial of $w\in\S^{d-1}$ of degree $2n,$ and 
$$
\tilde C_{n_o}(w) = \frac{2^{2n_o}}{(2n_o)!}\,D^{2n_o}f(0)[\underbrace{w,\ldots,w}_{\text{$2n_o$-times}}]
$$

Next consider 
$J(\varphi(y)).$ Inserting the Taylor series  
$$
\frac{1}{\sqrt{1-t}} = 1 + \frac{t}{2} + \frac{3t^2}{8} + \ldots,\qquad |t|<1,
$$
into  \eqref{jacoba} we obtain
\begin{equation}\label{j_ays}
J(\varphi(rw)) = 2^d\Big(1 +  \sum\limits_{n\ge1} \hat C_n(w)r^{2n}\Big), 
\end{equation}
where $\hat C_n:\S^{d-1}\to \R$ is a homogeneous symmetric polynomial of $w\in\S^{d-1}$ of degree $2n,$ and the series converges absolutely.

Now passing to polar coordinates by $y=rw$ in \eqref{I_1_a_aaa} and using \eqref{f_ays} and \eqref{j_ays} as well as the absolute convergence of the series we get 
\begin{align}
\low^*(z) =&  2^d\int_0^\gamma \frac{r^{d-1}}{4r^4-z}\, \Big(\sum\limits_{n\ge n_o} \int_{\S^{d-1}}  C_n(w)\,r^{2n}\Big)d\cH^{d-1}dr 
= \sum\limits_{n\ge n_o} \hat c_n \int_0^\gamma \frac{r^{2n+d-1}}{4r^4-z},\label{pizettiga_kel} 
\end{align}
where $C_n:\S^{d-1}\to \R$ is a homogeneous  polynomial of $w\in\S^{d-1}$ of degree $2n$ 
 and 
the coefficients $\{\hat c_n\}$ are real and given by
$$
\hat c_n:=2^d\int_{\S^{d-1}} C_n(w)d\cH^{d-1}. 
$$
Note that $\hat c_{n_o}=c_f,$ where $c_f$ is given by \eqref{c0_and_c1} and  the last series in \eqref{pizettiga_kel}
uniformly converges in any compact subset of $\C\setminus[0,4]$ since $\low^*$ and 
$$
z\in\C\setminus[0,4]\mapsto \singpart_{2n+d-1}(z):=\int_0^\gamma \frac{r^{2n+d-1}}{4r^4-z}
$$ 
are analytic functions in $\C\setminus[0,4]$ and all series in \eqref{pizettiga_kel} converge  pointwise\footnote{If $\{h_n\}$ is an equi-bounded sequence of  analytic functions in a connected open set $\Omega\subset \C$  converging pointwise to a function  $h:\Omega\to\C,$ then $h$ is analytic and $h_n$ converges uniformly to $h$ in compact subsets of $\Omega.$}.
By a direct computation one can readily check that for any $m\in\N_0,$  
there exist $c_m\in\R$  and an analytic function $f_m$ in the ball $W_{\gamma^4}(0)\subset\C$ such that for any $z\in (-\gamma^4,0),$
\begin{equation}\label{singpart_exp}
\singpart_m(z) = z^n\,\, \singpart_l^o(z) + c_m+
%\begin{cases}
z^\nu f_m((-z)^{1/2}), %& \text{$m$ is even},\\
%z^\nu f_m(-z), & \text{$m$ is odd}, 
%\end{cases} 
\end{equation}
where $n$ is the integer part of $m/4,$ i.e. $n:=[\frac{m}{4}],$  $l:=m-4n\in\{0,1,2,3\},$ 
$\nu=\frac12$ for $m=0,2$ and $\nu=1$ for $m=1,3$ or
$m\ge4,$
and 
\begin{equation}\label{singular_part_J}
\singpart_l^o (z):=
\begin{cases}
\frac{\pi}{4}\,(-z)^{-3/4} &\qquad  \text{if $\quad l=0$},\\
\frac{\pi}{8}\,(-z)^{-1/2}  &\qquad  \text{if $\quad l=1,$}\\
\frac{\pi}{8}\,(-z)^{-1/4}  & \qquad \text{if $\quad l=2$},\\
-\frac{1}{16}\,\ln(-z)  & \qquad \text{if $\quad l=3.$}
\end{cases}
\end{equation}
Inserting \eqref{singpart_exp} into \eqref{pizettiga_kel} we obtain 
\begin{align*}
\low^*(z) = 
\sum\limits_{n\ge n_o} \hat c_n\Big(z^{[\frac{2n+d-1}{4}]}\,  \singpart_{2n+d-1 - 4[\frac{2n+d-1}{4}]}^o &(z) + c_{2n+d-1} 
+\hat c_n (-z)^{\nu_n}f_{2n+d-1}((-z)^{1/2})\Big),
\end{align*}
where $\{c_{2n+d-1}\}\subset\R$ and  $\{f_{2n+d-1}\}$ is a sequence of analytic functions in $W_{\gamma^4}(0)$ and 
$$
\nu_n:=
\begin{cases}
\frac12, & 2n+d=1,3,\\
1, & \text{otherwise}.
\end{cases}
$$
Since \eqref{pizettiga_kel} converges locally uniformly in $\C\setminus[0,4],$ 
$$
C:=\sum\limits_{n\ge n_o} \hat c_n c_{2n+d-1}
$$
is finite and 
$$
\sum\limits_{n\ge n_o} \hat c_n (-z)^{\nu_n} f_{2n+d-1}((-z)^{1/2})= (-z)^\nu g((-z)^{1/2}),
$$
where  $g$ is analytic in $W_{\gamma^2}(0)$ and $\nu=\frac12$ for $2n_o+d=1,3$ and $\nu=1$ otherwise. Hence, 
\begin{align}
\low^*(z) =C+ (-z)^\nu g((-z)^{1/2}) +
\sum\limits_{n\ge n_o} \hat c_n\,z^{[\frac{2n+d-1}{4}]}\,  \singpart_{2n+d-1 - 4[\frac{2n+d-1}{4}]}^o (z),\label{first_attempt}
\end{align}

If $0\le 2n_o+d-1\le 3,$ then by \eqref{first_attempt},
\begin{align}
\low^*(z) =  C+ (-z)^{\nu}g((-z)^{1/2})  + &\hat c_{n_o} \singpart_{2n_o+d-1}^o(z)\no \\
+ &\sum\limits_{n\ge n_o+1} \hat c_n z^{[\frac{2n+d-1}{4}]}\, \singpart_{2n+d-1 - 4[\frac{2n+d-1}{4}]}^o(z).   \label{sdjh}
\end{align}
In view of \eqref{I_2_a_analytic} and the definition of $\singpart_l^o,$ from \eqref{sdjh} we obtain the expansions \eqref{l_f_asymp_d_odd} and \eqref{l_f_asymp_d_even} of $\low_f$ for $2n_o+d\le4.$ In particular, 
since $[\frac{2n+d-1}{4}]\ge1$ for $n\ge n_o+1,$ 
letting $z\to0$ in \eqref{sdjh} we get
\begin{equation}\label{int_cheksizlik}
\lim\limits_{z\to0} \low^*(z) =+\infty.
\end{equation}

If $2n_o+d-1\ge 4,$ then  
$[\frac{2n+d-1}{4}]\ge1$ for any $n\ge n_o.$ Therefore, by \eqref{first_attempt}, 
$\low^*(0):=\lim\limits_{z\to0} \low^*(z)$ exists and equals to $C.$ In particular, for $2n_o+d-1\le 7,$ one has
\begin{align}
\low^*(z) =  \low^*(0) -z g((-z)^{1/2})  + &\hat c_{n_o} z \singpart_{2n_o+d-1}^o(z) \no \\
+ &\sum\limits_{n\ge n_o+1} \hat c_n z^{[\frac{2n+d-1}{4}]}\, \singpart_{2n+d-1 - 4[\frac{2n+d-1}{4}]}^o(z), \label{yoyilmadi}  
\end{align}
from which and \eqref{I_2_a_analytic} we deduce the expansions \eqref{l_f_asymp_d_odd} and \eqref{l_f_asymp_d_even} of $\low_f$ for  
$5\le 2n_o+d\le 8.$ In particular, by virtue of \eqref{int_cheksizlik} and analyticity of $\low^{**}$ at $z=0,$   $\low_f$ is continous at $0$ if and only if $2n_o+d\ge 5.$ 
Notice also by \eqref{yoyilmadi} 
\begin{equation}\label{i_f_diffmas}
\lim\limits_{z\to 0 } \frac{\low^*(z) - \low^*(0)}{z} =+\infty,
\end{equation}
i.e. $\low^*$ (and hence $\low_f$) is not differentiable at $z=0.$ 

Finally, if $2n_o+d-1\ge8,$ then 
$[\frac{2n+d-1}{4}]\ge2$ for any $n\ge n_o.$ Therefore, by \eqref{first_attempt} there exists
$$
{\low^*}'(0):=\lim\limits_{z\to 0 } \frac{\low^*(z) - \low^*(0)}{z} =- g(0). 
$$
Now using the Taylor series of  $g$ at $0$ we get 
$$
zg((-z)^{1/2})= {\low^*}'(0) z +  z\sum\limits_{n\ge1}\frac{g^{(n)}(0)}{n!}\, (-z)^{n/2}.
$$
Inserting this in \eqref{first_attempt}, using the definition of $\singpart_l^o$ and the analyticity of $\low^{**}$ we get the expansions \eqref{l_f_asymp_d_odd} and \eqref{l_f_asymp_d_even} of $\low_f$ for  $2n_o+d\ge 9.$

By \eqref{int_cheksizlik} and \eqref{i_f_diffmas}, $\low_f$ is continously differentiable at $0$ if and only if $2n_o+d\ge9.$

Now the choice $\gamma=\frac{1}{\sqrt{2}}$ completes the proof.
\end{proof}

\begin{proposition} \label{prop:U_f_asymp}
Let $f:\T^d\to\R$ be a real-analytic function   
such that $q\in\T^d\mapsto f(\vec\pi+q)$ is even 
and 
\begin{equation}\label{f_assumption}
f(\vec\pi) = D^2f(\vec\pi) = \ldots =D^{2n_o-2}f(\vec\pi) = 0,\qquad D^{2n_o}(\vec\pi)\ne0 
\end{equation}
for some $n_o\in\N_0.$
Then:
\begin{itemize}
\item[--] $\low_f$ is continuous at $z=4d^2$ if and only if for $2n_o+d\ge 3,$  

\item[--] $\low_f$ is continuously differentiable  at $z=4d^2$ if and only if  for $2n_o+d\ge 5,$ in this case  
\begin{equation}\label{U_fning_hosilasi}
\low_f'(4d^2):=\int_{\T^d} \frac{f(q) dq}{(\dispersion(q)-4d^2)^2}= \lim\limits_{z\searrow4d^2} \int_{\T^d} \frac{f(q)dq}{(\dispersion(q) - z)^2}   
\end{equation}
exists.  
\end{itemize}
Moreover, if $z-4d^2\in (0,\frac{1}{16}),$  
$\low_f(z)$ is represented as:
\begin{itemize}
\item[(a)] if $d$ is odd, then 
\begin{equation}\label{asymp_u_f_d_odd}
\low_f(z) =
\begin{cases}
-\frac{\pi C_f}{\sqrt{z-4d^2}} + \sum\limits_{k\ge 0} a_k^d(z-4d^2)^{k/2},   & 2n_o+d=1,\\
\low_f(4d^2) + \pi C_f\,\sqrt{z-4d^2} + \sum\limits_{k\ge 2} a_k^d(z-4d^2)^{k/2},   & 2n_o+d=3,\\
\low_f(4d^2) + \low_f'(4d^2)\,(z-4d^2) + \sum\limits_{k\ge 3} a_k^d(z-4d^2)^{k/2},   & 2n_o+d\ge5;
\end{cases}
\end{equation} 

\item[(b)] if $d$ is even, then 
\begin{equation}\label{asymp_u_f_d_even}
\low_f(z) =
\begin{cases}
 C_f \ln\alpha  + \ln\alpha \sum\limits_{k\ge 1} b_k^d \alpha^k
+ \sum\limits_{k\ge0} c_k^d\alpha^k,  & 2n_o+d=2,\\
\low_f(4d^2) - C_f\alpha  \ln\alpha  + \ln\alpha \sum\limits_{k\ge 2} b_k^d \alpha^k
+ \sum\limits_{k\ge1} c_k^d\alpha^k,  & 2n_o+d=4,\\ 
\low_f(4d^2) + \low_f'(4d^2)\,\alpha + \ln\alpha \sum\limits_{k\ge 2} b_k^d \alpha^k
+ \sum\limits_{k\ge2} c_k^d\alpha^k, & 2n_o+d\ge6,
\end{cases}
\end{equation}  
\end{itemize}
where $\alpha:=z-4d^2,$ $\{a_k^d\},\{b_k^d\},\{c_k^d\}\subset\R$ and  
\begin{equation}\label{c_f_max}
C_f:= \frac{2^{2n_o+d-1}}{(8d)^{n_o+d/2}\,(2n_o)!}\,\int_{\S^{d-1}} D^{2n_o}f(\vec\pi)[w,\ldots,w]d\cH^{d-1}. 
\end{equation}
\end{proposition} 

\begin{proof} 
Since $4d^2-\dispersion(\cdot)$ has a unique non-degenerate minimum at $\vec\pi,$
the asymptotics of $\low_f(z)$ as $z\searrow4d^2$ can be done along the lines of, for instance, \cite[Lemma 4.1]{LH:2012}, hence we only sketch the proof. 
First using 
$$
\dispersion(q)^2-4d^2 =  - \sum (3 - \cos q_i)\sum (1+ \cos q_i)
$$
and 
the change of variables $q\mapsto \vec\pi +q$ in the integral, we represent $\low_f$ as
$$
\low_f(z) = -\int_{\T^d} \frac{f(\vec\pi +q)dq}{ \sum (3+ \cos q_i)\sum (1 - \cos q_i) + z- 4d^2}.%,\qquad \alpha: = z- 4d^2.
$$
For $\gamma\in(0,\frac{1}{\sqrt2})$ and $\varphi,$ given by \eqref{morse_map}, 
$$
\low_f(z) =- \up_1(z)-\up_2(z),
$$
where 
$$
\up_2(z): =  \int_{\T^d\setminus \varphi(B_\gamma(0))} \frac{f(\vec\pi +q)dq}{ \sum (3+ \cos q_i)\sum (1 - \cos q_i)+ z- 4d^2}
$$
is analytic in in the ball  $W_{\gamma/\sqrt2}(4d^2)\subset\C$ centered at $4d^2,$ since 
$$
 \sum (3+ \cos q_i)\sum (1 - \cos q_i)\ge 2(2d-\gamma^2)\gamma^2 \ge 2\gamma^2 
$$
(see e.g. the proof of analyticity of $\low^{**}$ in \eqref{I_2_a_analytic}), and 
$$
\up_1(z): =  \int_{\varphi(B_\gamma(0))} \frac{f(\vec\pi +q)dq}{ \sum (3+ \cos q_i)\sum (1 - \cos q_i)+z- 4d^2}.
$$

In  $\up_1,$ first using the change of variables $q=\phi(y)$  and then passing to polar coordinates $y=rw$ we get
$$
\up_1(z) = \frac14\int_0^\gamma \int_{\S^{d-1}} \frac{r^{d-1}\,f(\vec\pi + \phi(rw)) J(\phi(rw))d\cH^{d-1}dr}{(2d-r^2)r^2 +a^2},\qquad a:=\frac{z- 4d^2}{4}.
$$
Note that 
$$
\frac{1}{(2d-r^2)r^2 + a^2} =\frac{1}{2\sqrt{d^2+a^2}}\, \Big(\frac{1}{d+\sqrt{d^2+a^2} - r^2} + \frac{1}{r^2 + \sqrt{d^2+a^2} - d }\Big) 
$$
and $d^2+a^2 = \frac{z}{4},$
hence,
$$
\up_1(z) = \up_1^*(z)+ \up_1^{**}(z),
$$
where 
$$
\up_1^{**}(z):=\frac{1}{2\sqrt z}\, \int_0^\gamma \int_{\S^{d-1}} \frac{r^{d-1}\,f(\vec\pi + \phi(rw)) J(\phi(rw))d\cH^{d-1}dr}{2d+\sqrt{z} - 2r^2}
$$
is analytic in $W_\gamma(4d^2)$ and 
$$
\up_1^*(z):=\frac{1}{4\sqrt{z}}\, \int_0^\gamma \int_{\S^{d-1}} \frac{r^{d-1}\,f(\vec\pi + \phi(rw)) J(\phi(rw))d\cH^{d-1}dr}{r^2 + \frac{\sqrt{z} - 2d}{2}}.
$$
Since $f$ analytic and  $y\mapsto f(\vec\pi + y)$ is even 
from the Taylor series of $f$ at $\vec\pi,$ \eqref{f_assumption} and \eqref{taylor_arcsin}  we obtain 
\begin{equation}\label{f_ays_1}
f(\vec\pi + \phi(rw)) = \sum\limits_{n\ge n_o} \frac{1}{(2n)!}\,D^{2n}f(\vec\pi)[\phi(rw),\ldots,\phi(rw)] =\sum\limits_{n\ge n_o} \tilde C_n(w)r^{2n},
\end{equation}
where $\tilde C_n:\S^{d-1}\to\R$ is a homogeneous polynomial of $w\in\S^{d-1}$ of degree $2n$ and 
$$
\tilde C_{n_o} = \frac{2^{2n_o}}{(2n_o)!}D^{2n_o}f(\vec\pi)[w,\ldots,w]
$$
Inserting \eqref{f_ays_1} and the expansion \eqref{j_ays} of $J(\phi(rw))$ in $\up_1^*$ we get 
\begin{align}\label{gahets}
\up_1^*(z):= 
\frac{2^{d-2}}{\sqrt z}\, & \sum\limits_{n\ge n_o} \tilde c_n \int_0^\gamma  \frac{r^{d-1+2n} dr} {r^2 + \frac{\sqrt z - 2d}{2}},
\end{align}
where $\{\tilde c_n\}\subset \R$ and $\tilde c_{n_o} = \int_{\S^{d-1}} \tilde C_{n_o}(w)d\cH^{d-1}.$ 
Note that 
$
\frac{\sqrt z - 2d}{2} = \frac{z - 4d^2}{2\sqrt z + 4d}
$ 
%and 
and by \cite[Eq. 4.3]{LH:2012} for any $m\in\N_0,$
\begin{equation}\label{asdas}
\int_0^\gamma \frac{r^m dr}{r^2 + \frac{\sqrt z - 2d}{2}}  = \hat g_m(z) +
\begin{cases}
\frac{\pi\sqrt{2\sqrt z + 4d}}{2\sqrt{z - 4d^2}}\,(-\frac{z - 4d^2}{2\sqrt z + 4d})^k  & \text{if $m=2k,$}\\[2mm]
-\frac12\,(-\frac{z - 4d^2}{2\sqrt z + 4d})^k \ln(z - 4d^2)  &\text{if $m=2k+1,$} 
\end{cases} 
\end{equation}
where  $\hat g(\cdot)$ is analytic in $W_{\gamma^2}(0).$
Then using the analyticity of $\sqrt{z}$ and $\frac{1}{2\sqrt z+4d}$ in $W_d(4d^2),$  we rewrite \eqref{gahets} as 
$$
\up_1^*(z) = \frac{(4d^2-z)^{n_o+(d-1)/2}}{\sqrt{z-4d^2}}\,  \sum\limits_{k\ge 0} \tilde a_k(z-4d^2)^{k} +\sum\limits_{k\ge 0} \hat a_k(z-4d^2)^{k} 
$$
if $d$ is odd, and 
$$
\up_1^*(z)  = - (4d^2-z)^{n_o+d/2-1}\,\ln(z-4d^2) \sum\limits_{k\ge 0} \tilde b_k (z-4d^2)^k
+ \sum\limits_{k\ge0} \hat b_k(z-4d^2)^k 
$$
if $d$ is even, where $\{\tilde a_k\},\{\tilde b_k\},\{\hat a_k\},\{\hat b_k\}\subset\R,$ $\tilde a_0=\pi C_f$ $\tilde b_0=C_f,$ and all series absolutely converge in $W_{\gamma^2/4}(4d^2)\subset\C$ (for $\sqrt{\cdot}$-function we choose the branch satisfying $\sqrt{1}=1$).  
Since 
$$
\low_f(z) = -\up_1^*(z) - \up_1^{**}(z) - \up_2(z)
$$ 
and $ \up_1^{**}(\cdot)$ and $\up_2(\cdot)$ are analytic in $W_{\gamma}(4d^2),$ 
\begin{equation}\label{u_f_d_odd}
\low_f(z) = -\frac{(4d^2-z)^{n_o+(d-1)/2}}{\sqrt{z-4d^2}}\,  \sum\limits_{k\ge 0} \tilde a_k(z-4d^2)^{k} +\sum\limits_{k\ge 0}  a_k(z-4d^2)^{k}
\end{equation}
if $d$ is odd, and 
\begin{equation}\label{u_f_d_even}
\low_f(z)  = (4d^2-z)^{n_o+d/2-1}\,\ln(z-4d^2) \sum\limits_{k\ge 0} \tilde b_k (z-4d^2)^k
+ \sum\limits_{k\ge0} b_k(z-4d^2)^k 
\end{equation}
if $d$ is even, where $\{a_k\},\{b_k\} \subset\R,$  and supposing that $\sqrt{1}=1,$ 
all series absolutely converge in $W_{\gamma^2/4}(4d^2)\subset\C$ . 

{\it Case $2n_o+d=1,2.$} By \eqref{u_f_d_odd} and \eqref{u_f_d_even} we get 
$$
\low_f(z) = -\frac{\pi C_f}{\sqrt{z-4d^2}} + \sum\limits_{k\ge 0} a_k^d(z-4d^2)^{k/2}, 
$$
and 
$$
\low_f(z)  =  C_f \ln(z-4d^2)  + \ln(z-4d^2) \sum\limits_{k\ge 1} b_k^d (z-4d^2)^k
+ \sum\limits_{k\ge0} c_k^d(z-4d^2)^k, 
$$
where $\{a_k^d\},\{b_k^d\},\{c_k^d\}\subset\R.$ In particular, 
$$
C_f\,\lim\limits_{z\searrow 4d^2} \low_f(z) =-\infty.
$$

{\it Case $2n_o+d=3,4.$}
Notice that if $2n_o+d\ge3,$ then by \eqref{u_f_d_odd} and \eqref{u_f_d_even}, the limit 
$$
\low_f(4d^2): = \lim\limits_{z\searrow 4d^2} \low_f(z) 
$$
exists, is finite and equal to $a_0$ if $d$ is odd and to $c_0$ if $d$ is even. In particular, 
if $2n_o+d=3,$ then 
\begin{equation}\label{asasda}
\low_f(z) = \low_f(4d^2) + \pi C_f\,\sqrt{z-4d^2} + \sum\limits_{k\ge 2} a_k^d(z-4d^2)^{k/2}  
\end{equation}
and if $2n_o+d=3,$ then 
\begin{equation}\label{daadfaf}
\low_f(z)  =  \low_f(4d^2) -(z-4d^2) C_f \ln(z-4d^2)  + \ln(z-4d^2) \sum\limits_{k\ge 2} b_k^d (z-4d^2)^k
+ \sum\limits_{k\ge1} c_k^d(z-4d^2)^k  
\end{equation}
for some $\{a_k^d\},\{b_k^d\},\{c_k^d\}\subset\R.$ 
Since all series in \eqref{asasda} and \eqref{daadfaf} locally uniformly converge in $(4d^2,4d^2+\gamma^2/4),$ we can differentiate them term-by-term; in particular, 
$$
C_f\,\lim\limits_{z\searrow 4d^2} \low_f'(z) =-\infty.
$$

{\it Case $2n_o+d\ge5.$} In this case by \eqref{u_f_d_odd} and \eqref{u_f_d_even} 
$$
\low_f'(4d^2):=\lim\limits_{z\to 4d^2} \frac{\low_f(z) - \low_f(4d^2)}{z-4d^2} = \begin{cases}
                                                                          a_1, & \text{$d$ odd,}\\
                                                                          b_1, & \text{$d$ even,}\\
                                                                         \end{cases}
$$
thus,
$$
\low_f(z) = \low_f(4d^2) + \low_f'(4d^2)\,(z-4d^2) + \sum\limits_{k\ge 3} a_k^d(z-4d^2)^{k/2}  
$$
if $2n_o+d$ is odd and 
$$
\low_f(z)  =  \low_f(4d^2) + \low_f'(4d^2)\,(z-4d^2) + \ln(z-4d^2) \sum\limits_{k\ge 2} b_k^d (z-4d^2)^k
+ \sum\limits_{k\ge2} c_k^d(z-4d^2)^k  
$$
if $2n_o+d$ is even,
where $\{a_k^d\},\{b_k^d\},\{c_k^d\}\subset\R.$ 

Now the choice $\gamma=\frac12$ completes the proof.
\end{proof}

\begin{remark}\label{rem:asymp_low_f}
When 
$$
|\hat v(x)| = O(|x|^{2n_0+d+1})\qquad \text{as $|x|\to \infty$}
$$
for some $n_0\ge1,$ 
one has $v\in C^{2n_0}(\T^d).$ In this case the Taylor series \eqref{f_anal_qator} and \eqref{f_ays_1}  of $f$ becomes only asymptotics of order $2n_0-1$ and thus, 
instead of expansions \eqref{l_f_asymp_d_odd}-\eqref{l_f_asymp_d_even} and   \eqref{asymp_u_f_d_odd}-\eqref{asymp_u_f_d_even} of $\low_f$ one has only asymptotics up to order $2n_0-1.$
\end{remark}

\subsection*{Acknowledgements} 

Sh. Kholmatov acknowledges support from the Austrian Science Fund (FWF) project
M 2571-N32.


\begin{thebibliography}{oo}
\bibitem{ALM:04:Puan} S. Albeverio, S. Lakaev, Z. Muminov:
Schr\"{o}dinger  operators on lattices. The Efimov effect and discrete spectrum
asymptotics. Ann. Inst. H. Poincar\'e Phys. Theor. {\bf 5} (2004), 743--772.

\bibitem{ALMM:2006}   S. Albeverio, S. Lakaev, K. Makarov, Z. Muminov:
The threshold effects for the two-particle Hamiltonians on lattices.
Commun. Math. Phys. {\bf 262} (2006), 91--115.



\bibitem{AP:1986} { A. Andrew, J. Paine:} Correction of finite element estimates for Sturm-Liouville eigenvalues. Numer. Math. {\bf50} (1986),  205--215.


\bibitem{BT:2017} G. Basti, A. Teta: Efimov effect for a three-particle 
system with two identical  fermions. Ann. Henri Poincar\'e {\bf 18} (2017), 
3975--4003.

\bibitem{BK:2019} { M. Ben-Artzi, G. Katriel:}
Spline functions, the biharmonic operator and approximate eigenvalues.
Numer. Math. {\bf 141} (2019),  839--879.

\bibitem{Bou:2003} { A. Boumenir:} Sampling for the fourth-order Sturm-Liouville differential operator. J. Math. Anal. Appl. {\bf278} (2003), 542--550. 

\bibitem{DKV:2018} S. Dipierro, A. Karakhanyan, E. Valdinoci: A free boundary problem driven by the biharmonic operator.  arXiv:1808.07696v2 [math.AP].


\bibitem{Gal:2008}  { I. Gali\'c {\it et al}:} 
Image compression with anisotropic diffusion.
J. Math. Imaging Vis.  {\bf31} (2008), 255--269.

\bibitem{GSch:97} { G. Graf, D. Schenker:} 2-magnon scattering in 
the Heisenberg model. Ann. Inst. Henri Poincar\'e, Phys. Th\'eor. {\bf 67} (1997),
91--107.

\bibitem{GHKW:2018} { J. Graef, Sh. Heidarkhani, L. Kong, M. Wang:} Existence of solutions to a discrete fourth order boundary value problem. J. Difference Equ. Appl. {\bf 24} (2018),  849--858

\bibitem{Gr:2014} D. Gridnev: Three resonating fermions in flatland: 
proof of the super Efimov effect and the exact discrete spectrum asymptotics.
J. Phys. A: Math. Theor. {\bf 47} (2014).


\bibitem{HHV:2018} V. Hoang, D. Hundertmark, J. Richter, S. Vugalter:  Quantitative bounds versus existence of weakly coupled bound states for Schr\"odinger type operators. arXiv:1610.09891 [math-ph].


\bibitem{HPW:2015} { S. Hoffmann, G. Plonka, J. Weickert:} Discrete green’s functions for harmonic and biharmonic inpainting with sparse atoms. {\it In:} X. Tai {\it et al} (eds) Energy Minimization Methods in Computer Vision and Pattern Recognition. EMMCVPR 2015. Lecture Notes in Computer Science, vol 8932 (2015). Springer, Cham.


\bibitem{JBC:1998} { D. Jaksch {\it et al}:}
Cold bosonic atoms in optical lattices. Phys. Rev. Lett. {\bf 81}
(1998), 3108--3111.
 
\bibitem{HP:2019} Sh. Kholmatov, M. Pardabaev:  Asymptotics of eigenvalues of the zero-range perturbation of the discrete bilaplacian. arXiv:1909.11789 [math.SP].

 
\bibitem{KS:1980} M. Klaus, B. Simon: 
Coupling constant thresholds in nonrelativistic quantum mechanics. I. Short-range two-body case. Ann. Phys.
{\bf 130} (1980), 251--281.

\bibitem{Lak:1993} S. Lakaev: The Efimov effect of a system of three
identical quantum lattice particles. Funkcional. Anal.
Prilozhen. {\bf 27} (1993), 15--28.  



\bibitem{LKL:2012} { S. Lakaev, A. Khalkhuzhaev, Sh. Lakaev:}
Asymptotic behavior of an eigenvalue of the two-particle
discrete Schr\"odinger operator.
Theoret. Math. Phys. {\bf 171}  (2012), 800--811.

\bibitem{LH:2011} { S. Lakaev, Sh. Kholmatov:}
 Asymptotics of eigenvalues of two-particle
Schr\"odinger operators on lattices with zero range
interaction. J. Phys. A: Math. Theor. {\bf 44} (2011). 

\bibitem{LH:2012} S. Lakaev, Sh. Kholmatov: Asymptotics of the eigenvalues of a discrete Schr\"odinger operator with zero-range potential. Izvestiya: Mathematics {\bf76} (2012), 946--966. 

\bibitem{LSA:2012} { M. Lewenstein, A. Sanpera, A. Ahufinger:}
 Ultracold Atoms in Optical Lattices.
Simulating Quantum Many-Body Systems. 
Oxford University Press, Oxford, 2012.

\bibitem{MZ:2016} R. Mardanov, S. Zaripov: Solution of Stokes flow problem using biharmonic equation formulation and multiquadrics method. Lobachevskii J. Math. {\bf 37} (2016), 268--273. 



\bibitem{Mat} { D. Mattis:} {The few-body problem
on a lattice,} Rev. Mod. Phys. {\bf 58:2} (1986),
361--379.

\bibitem{MW:1987} P. McKenna,  W. Walter: Nonlinear oscillations in a suspension bridge. Arch. Rational Mech. Anal. {\bf 98} (1987), 167--177. 


\bibitem{Mog} { A. Mogilner:} Hamiltonians in solid-state physics
as multiparticle discrete Schr\"odinger
operators: problems and results.
Adv. in Sov. Math. {\bf 5} (1991), 139--194.


\bibitem{NE:2017} P. Naidon,  S. Endo: Efimov physics: a review. 
Rep. Prog. Phys. {\bf 80} (2017).

\bibitem{RB:2013} { A. Rattana, C. B\"ockmann:} Matrix methods for computing eigenvalues of Sturm-Liouville problems of order four. J. Comput. Appl. Math. {\bf249} (2013), 144--156. 

\bibitem{S:1977} {B. Simon:} Notes on infinite determinants of Hilbert space operators. Adv. Math. {\bf24} (1977), 244--273.


\bibitem{Sob:1993} A. Sobolev: The Efimov effect. Discrete spectrum 
asymptotics. Commun. Math. Phys. {\bf 156} (1993), 127--168.

\bibitem{Tam:1991}  H. Tamura: The Efimov effect of three-body 
Schr\"odinger operator. J. Funct. Anal. {\bf 95}, 433--459  (1991).

\bibitem{Tee:1963} { G. Tee:} A novel finite-difference approximation to the biharmonic operator.
The Computer Journal {\bf 6} (1963), 177--192.


\bibitem{Yaf:1974}  D. Yafaev: On the theory of the discrete 
spectrum of the three-particle Schr\"odinger operator.
Math. USSR-Sb. {\bf 23} (1974), 535--559.

\bibitem{Wall:2015} { M. Wall:} Quantum Many-Body Physics of Ultracold 
Molecules in Optical Lattices. Models and Simulation Models. 
Springer Theses, Cham-Heidelberg-New York, 2015.

\bibitem{Wink:2006} K. Winkler {\it et al.}: Repulsively bound atom 
pairs in an optical lattice. Nature {\bf 441} (2006),  853--856.




\end{thebibliography}
\end{document}